\documentclass[preprint,12pt]{elsarticle}

\usepackage{amsfonts,amssymb,amsmath,amsthm,bm}
\usepackage{mathtools}
\usepackage{dsfont}
\usepackage{color}
\usepackage{geometry}                		
\usepackage{subfigure}

\newtheorem{lemma}{Lemma}
\newtheorem{theorem}{Theorem}

\newtheorem{proposition}{Proposition}

\usepackage{lineno}

\makeatletter
\def\ps@pprintTitle{%
  \let\@oddhead\@empty
  \let\@evenhead\@empty
  \let\@oddfoot\@empty
  \let\@evenfoot\@oddfoot
}
\makeatother

\begin{document}

\begin{frontmatter}

\title{Multi-type branching diffusions with small mutation rates}
\author[label1]{Conrad J.\ Burden}
\ead{conrad.burden@anu.edu.au}
\address[label1]{Mathematical Sciences Institute, Australian National University, Canberra, Australia}


\begin{abstract}
Approximate solutions are found for Feller-like neutral multi-type branching diffusions $\big{(}\mathbf{X}(t)\big{)}_{t \in \mathbb{R}_{\ge 0}}$ 
in the limit of small mutation rates.  The method 
employed involves solving approximations to the Laplace transformed forward Kolmogorov equation by integrating along characteristics.  
To leading order in the scale $\theta$ of the overall mutation rate the super-critical diffusion is found to collapse onto a line density aligned 
with the stationary left eigenvector of the rate matrix following a rapid change of behaviour at a critical time $t_{\rm c}$, which has a weak logarithmic 
dependence on $\theta$.  First order in $\theta$ approximations to the density and moments of $\mathbf{X}(t)$ are also determined.  The first-order 
approximation corresponds to allowing at most one mutation in the coalescent tree of a sample and is valid for $t < t_{\rm c}$ in the supercritical 
case, and more generally for the critical and sub-critical cases.  
\end{abstract}

\end{frontmatter}

%

\section{Introduction}

In population genetics, approximate solutions to diffusion models with 
small mutation rates can be useful for estimating rate matrices from site frequency data~\citep{SchrempfHobolth17,BurdenTang17}.  
Using data from \citet[Figure 3b]{Lynch16} of per nucleotide site per generation mutation rates in eukaryote genomes, \citet[Section~1]{VOGL2020106} 
estimate  diffusion-scaled mutation rates to be typically $< 10^{-2}$.  Calculations exploiting this approach to date have been based on 
Moran~\citep{VoglClemente12} or Wright-Fisher~\citep{BurdenTang16} models.  There is value in extending this programme to diffusions based on 
branching models, which allow for stochastically variable population sizes.  

The purpose of this paper is to explore time-dependent solutions to multi-type branching diffusions in the limit of small mutation rates.  
Specifically we are interested in the process  
$\big{(}\mathbf{X}(t)\big{)}_{t \in \mathbb{R}_{\ge 0}} = \big{(} (X_1(t), \ldots, X_d(t))^{\rm T}\big{)}_{t \in \mathbb{R}_{\ge 0}} $ 
with $d \ge 1$ types and generator  
\begin{equation}		\label{genOfXMultitype}
{\cal L}_\mathbf{X} = \frac{1}{2} \sum_{i = 1}^d x_i \frac{\partial^2}{\partial x_i^2} + \alpha \sum_{i = 1}^d x_i \frac{\partial}{\partial x_i}
				+ \sum_{i , j= 1}^d  \gamma_{ji} x_j \frac{\partial}{\partial x_i},
\end{equation}
where $\alpha \in \mathbb{R}$ and $(\gamma_{ij})_{i, j = 1}^d$ is an instantaneous rate matrix satisfying $\gamma_{ij} \ge 0, i\ne j$, 
$\sum_{j = 1}^d \gamma_{ij} = 0$.  The diffusion is termed sub-critical, critical or super-critical according as $\alpha < 0$, $\alpha = 0$ 
or $\alpha > 0$ respectively.  For $d > 1$ there is no known exact time-dependent solution for the process generated by Eq.~(\ref{genOfXMultitype}) 
for a general rate matrix $\gamma$ and given initial condition $\mathbf{X}(0)$. 

This process arises in population genetics as the diffusion limit of a a multi-type Bienyam\'{e}-Galton-Watson (BGW)
process~\citep[Chapter~2]{Harris64}.  Here, diffusion limit 
is meant in the sense of \citet{kimura1964diffusion} in that simultaneous limits are taken in which the continuum time 
between generations is taken to zero, the effective population size becomes infinite, and the per-generation mutation rate is taken to zero in such a 
way that mutation events along any lineage become a continuous-time Markov process with a finite rate matrix 
$\gamma$~\citetext{\citealp{burden2018mutation}; \citealp[Section~2]{,BurdenGriffiths23}}.  Eq.~(\ref{genOfXMultitype}) can also 
be interpreted as a diffusion limit of a multi-type linear birth-death (BD) process in which birth and death rates become infinite while their difference 
$\alpha$ remains finite, by generalising the derivation for a Feller diffusion in \citet[Section~2.1]{BurdenGriffiths24}.  The process is neutral 
(as opposed to selective) in the sense that the expected total number of offspring per parent in the underlying BGW process or birth rate in the 
underlying BD process, and hence the parameter $\alpha$, does not depend on the parent type.  

If $d = 1$ the final term is not present in Eq.~\ref{genOfXMultitype} and the generator reduces to that of the well known Feller diffusion~\citep{feller1951diffusion}.  
Because of neutrality the total population $\big{(}\sum_{i = 1}^d X_i(t)\big{)}_{t \in \mathbb{R}_{\ge 0}}$ of the multi-type process is itself a 
rate-$\alpha$ Feller diffusion.  As $t \to \infty$ a super-critical Feller diffusion with initial condition $X(0) = x_0$ becomes extinct with probability 
$e^{-2\alpha x_0}$, and, if conditioned on survival, asymptotes to a known $x_0$-dependent density.  The population of a critical or sub-critical Feller 
diffusion will become extinct almost surely as $t \to \infty$.  In this limit the distribution of interest is the quasi-stationary 
distribution~\citep{yaglom1947certain,lambert2007quasi}, that is, the distribution of $X(t)$ conditioned on survival.  For $\alpha = 0$ the limiting distribution of 
$t^{-1}X(t)|(X(t) > 0)$ as $t \to \infty$ is exponential with rate 2, and for $\alpha < 0$, the limiting distribution of $X(t)|(X(t) > 0)$ as $t \to \infty$ is 
exponential with rate $2|\alpha|$.  

Known asymptotic results for the stationary and quasi-stationary behaviours of discrete multi-type branching processes are listed in 
\citet[p44]{Harris64} and \citet[pp186-192]{Athreya_1972} and summarised in the introduction to \citet{BurdenGriffiths23}.  
Rigorous results for the asymptotic stationary behaviours for critical and supercritical multi-type diffusions can be found in the continuous-state branching 
process literature~\citep{Barczy24,champagnat2008limit,Kyprianou18}.  Relatively straightforward derivations specific to the generator Eq.~(\ref{genOfXMultitype})   
are given in \citet[Section~4]{BurdenGriffiths23}.  The quasi-stationary and stationary distributions respectively are seen to collapse onto the line densities 
associated with a Feller diffusion lying along a ray in $(\mathbb{R}_{\ge 0})^d$ aligned with the principle left eigenvector of the mutation rate matrix $\gamma$.  
This is consistent with the known asymptotic limits of the discrete processes.  

The quasi-stationary limit of the subcritical multi-type diffusion is less straightforward \citep[Section~3.2]{champagnat2008limit}.  
In this case the limiting density in in $(\mathbb{R}_{\ge 0})^d$ does not collapse onto a line, and in common with the discrete case, is independent of the initial 
condition~$\mathbf{X}(0)$~\citep[Chapter V.4]{Athreya_1972}.  Although an exact solution of the limiting asymptotic density of the diffusion is 
not known, the asymptotic mean of the population is known to be $(2|\alpha|)^{-1}\bm{\pi}$, 
where $\bm{\pi}$ is the left-stationary eigenvector of the rate matrix $\gamma$ normalised so that $\sum_{i = 1}^d \pi_i = 1$ \citep[Lemma~2]{BurdenGriffiths23}.

We begin our analysis in Section~\ref{sec:backgound} with a brief summary of properties of the known solution to the $d = 1$ Feller diffusion needed for 
our subsequent analysis of small mutation rates.  

An obvious starting point for a small mutation rate approximation when $d > 1$ is to take as the unperturbed 
density the solution generated by Eq.~(\ref{genOfXMultitype}) when the rate matrix $\gamma$ is set identically equal to zero.  The naive unperturbed solution 
then corresponds trivially to $d$ independent Feller diffusions with given initial conditions $X_1(0), \ldots, X_d(0)$.   
However, this raises an interesting problem in the supercritical case, for which the asymptotic density as $t \to \infty$ is a line density aligned with the 
principle left eigenvector of the rate matrix, and it does not seem possible to approximate this known asymptotic density as a perturbation 
of the naive unperturbed solution.  The resolution to the dilemma was determined for the $d = 2$ case by \citet{burden2018mutation}, who 
found that for very small mutation rates the density of $\mathbf{X}(t)$ undergoes a sudden change of behaviour at a critical time $t_{\rm c}$ at which 
it transitions from approximating the naive unperturbed solution to being sandwiched between two lines which collapse towards the principle left 
eigenvector of the rate matrix.  In Section~\ref{sec:leadingOrder} their analysis is generalised to arbitrary $d \ge 2$.  

In Section~\ref{sec:firstOrder} the solution to the forward Kolmogorov is found to the first order in the scale of the rate matrix $\gamma$ via 
a Laplace transform by perturbing the naive unperturbed solution.  It is noted that for the super-critical the solution only applies for $t < t_{\rm c}$.  
In Section~\ref{sec:XMoments} this solution is used to calculate moments to first order in mutation rates of $\mathbf{X}(t)$.  We also discuss the consistency of the 
calculated moments in the current paper with earlier results for the quasi-stationary limit of the moments of the sub-critical diffusion by 
\citet{BurdenGriffiths23}.  Conclusions are drawn in 
Section~\ref{sec:Conclusions}.  Most of the technical details and proofs are contained in Appendices~A to E.

%

\section{Definitions and background material}
\label{sec:backgound}

A convenient parameterisation for the rate matrix is 
\begin{equation}	\label{gammaParam}
\gamma_{ij} = \tfrac{1}{2} \theta Q_{ij}, 
\end{equation}
where we require $-1 \le \min_{i = 1, \ldots, d} Q_{ii} \le 0$, and hence any $\theta \ge 2\max_{i = 1, \ldots, d} \sum_{j: j \ne i}\gamma_{ij}$ 
becomes a suitable perturbation parameter for studying slow mutation rates.  With this parameterisation mutations along any ancestral branch 
can be thought of as a finite-state discrete-time Markov chain with transition matrix $(\delta_{ij} + Q_{ij})_{i, j = 1}^d$, whose transitions occur 
at times determined by a Poisson process with rate $\tfrac{1}{2}\theta$.  Note that this interpretation includes ghost mutations of a type to itself.  

For $d = 1$ the forward equation is that for a Feller diffusion \citep[Eq.~(5.1)]{feller1951diffusion}, 
\[
\frac{\partial f_X}{\partial t} = \frac{1}{2} \frac{\partial^2}{\partial x^2} (x f_X) - \alpha \frac{\partial}{\partial x} (x f_X).
\]
The solution with initial condition $X(0) = x_0$ is obtained via a Laplace transform (LT) in~$x$ \citep[Section~5.11]{Cox78}, yielding \citep[Eq.~(5.2)]{feller1951diffusion},
\begin{eqnarray}	\label{FellersPoissonGammaMixture}
f_X(t, x; \alpha, x_0) & = &  \delta(x) e^{-x_0\mu(t; \alpha)}  \\
	& & +\, \sum_{l = 1}^\infty e^{-x_0\mu(t; \alpha)} \frac{(x_0\mu(t; \alpha))^l}{l!} \frac{x^{l - 1}}{\beta(t; \alpha)^l (l - 1)!} e^{-x/\beta(t; \alpha)}, \quad x \ge 0,  \nonumber
\end{eqnarray}
where $\delta(x)$ is the Dirac $\delta$-function \citep[Section~15]{Dirac58} and  
\begin{equation}	\label{muBetaDef}
\mu(t; \alpha) = \frac{2\alpha e^{\alpha t}}{e^{\alpha t} - 1}, \qquad \beta(t; \alpha) = \frac{e^{\alpha t} - 1}{2\alpha}.  
\end{equation}
We set $\mu(t; 0) = 2/t$ and $\beta(t; 0) =  t/2$.  The first term in Eq.~(\ref{FellersPoissonGammaMixture}) accounts for extinction of the entire population up to time $t$.  
The interpretation of the remaining, continuous part is that the population at time $t$ is composed of independent 
families descended from $L = l$ founders at time $0$, where $L$ is Poisson with mean $x_0\mu(t; \alpha)$.
The size of each family is an independent exponential random variable with means $\beta(t; \alpha)$, 
as observed by \citet[Theorem~2.1(ii)]{o1995genealogy}.  See also  \citet[Section~2.2 and Fig.~1]{BurdenGriffiths24}.  

%
%

\subsection{Exponential scaling}

Throughout much of the current paper we will work with an exponentially scaled multi-type diffusion defined by 
\begin{equation}	\label{ZDef}
\mathbf{Z}(t) = \mathbf{X}(t) e^{-\alpha t}.  
\end{equation}
The corresponding generator is \citep[Eq.~(15)]{BurdenGriffiths23} 
\begin{equation}	\label{genOfZ}
{\cal L}_\mathbf{Z} = \tfrac{1}{2} e^{-\alpha t}\sum_{i = 1}^d z_i \frac{\partial^2}{\partial z_i^2} 
				+ \sum_{i , j= 1}^d  \gamma_{ji} z_j \frac{\partial}{\partial z_i}.
\end{equation}
Define the LT 
\[
\zeta(t; {\bm \phi}) = \mathbb{E}[e^{-{\bm \phi}\cdot\mathbf{Z}(t)}].
\]
Then 
\begin{eqnarray*}
{\cal L}_\mathbf{Z} e^{-{\bm \phi}\cdot\mathbf{z}} & = & \tfrac{1}{2} e^{-\alpha t}\sum_{i = 1}^d z_i \phi_i^2 e^{-{\bm \phi}\cdot\mathbf{z}} - 
						\sum_{i , j= 1}^d  \gamma_{ji} z_j \phi_i e^{-{\bm \phi}\cdot\mathbf{z}}	\\
		& = & \left( - \tfrac{1}{2} e^{-\alpha t}\sum_{i = 1}^d  \phi_i^2 \frac{\partial}{\partial \phi_i} 
						+ \sum_{i , j= 1}^d  \gamma_{ji} \phi_i \frac{\partial}{\partial \phi_j} \right) e^{-{\bm \phi}\cdot\mathbf{z}},
\end{eqnarray*}
and hence
\begin{equation}	\label{zetapde}
\frac{\partial}{\partial t} \zeta(t; {\bm \phi}) = \left( - \tfrac{1}{2} e^{-\alpha t}\sum_{i = 1}^d  \phi_i^2 \frac{\partial}{\partial \phi_i} 
						+ \sum_{i , j= 1}^d  \gamma_{ji} \phi_i \frac{\partial}{\partial \phi_j} \right) \zeta(t; {\bm \phi}).  
\end{equation}

For the case of $d = 1$ allele the LT satisfies 
\[
\frac{\partial \zeta_{\text{1-allelle}}(t, \phi)}{\partial t}  + \tfrac{1}{2} e^{-\alpha t} \phi^2 \frac{\partial \zeta_{\text{1-allelle}}(t, \phi)}{\partial \phi} = 0.  
\]
The function $\zeta_{\text{1-allelle}}$ is constant along characteristic curves satisfying 
\begin{equation}	\label{charEq1Allele}
\frac{d\phi}{dt} = \tfrac{1}{2} e^{-\alpha t} \phi^2, 
\end{equation}
with initial condition $\phi(t = 0) = \xi$, say.  The solution is 
\[
\phi = \frac{\xi}{1 - \frac{\xi}{2\alpha}(1 - e^{-\alpha t})} = \frac{\xi}{1 - \xi/\mu(t)}, 
\]
which rearranges to 
\[
\xi = \frac{\phi}{1 + \phi/\mu(t)}. 
\]
If $Z(0) = X(0) = x_0$, we have $\zeta_{\text{1-allelle}}(t; \phi) = \zeta_{\text{1-allelle}}(0; \xi) = \mathbb{E}[e^{-\xi X(0)}] = e^{-\xi x_0}$, yielding the solution 
\begin{equation}	\label{zeta1Allele}
\zeta_{\text{1-allelle}}(t, \phi; x_0) = \exp\left\{ -x_0 \frac{\phi}{1 + \phi/\mu(t)} \right\}.  
\end{equation}
This is the LT of 
\begin{eqnarray}	\label{f1Allele}
f_{\text{1-allelle}}(t, z; x_0) & = & e^{\alpha t} f_X(t, ze^{\alpha t}; x_0) \nonumber \\
	& = & \delta(z) e^{-x_0\mu(t)} + \, \sum_{l = 1}^\infty e^{-x_0\mu(t)} \frac{(x_0\mu(t))^l}{l!} \frac{\mu(t)^l x^{l - 1} e^{-z \mu(t)}}{(l - 1)!} , \quad z \ge 0, \nonumber \\
\end{eqnarray}
with $f_X(t, x; x_0)$ the Poisson-Gamma mixture in Eq.~(\ref{FellersPoissonGammaMixture}).  To facilitate calculation of LTs in the following section 
we also define $f_{\text{1-allelle}}(t, z; x_0)$ to be zero for $z < 0$.  
Conditional on a population descended from a single ancestor surviving to time $t$, we have \citep[Lemma~2]{BurdenGriffiths25}
\begin{equation}	\label{1ancestor1allele}
f_{\text{1-allele}}^{\text{1-ancestor}}(t, z) = \lim_{x_0 \to 0} \frac{f_{\text{1-allelle}}(t, z; x_0) - \delta(z) e^{-x_0\mu(t)}}{1 - e^{-x_0\mu(t)}}
	= \mu(t) e^{-z \mu(t)}.
\end{equation}

%
%

\section{Leading order in $\theta$ solution for $d$ types and $\alpha > 0$}
\label{sec:leadingOrder}

 \citet[Theorem~1.4]{Kyprianou18} and \citet[Proposition~1]{BurdenGriffiths23} have shown that the asymptotic density as $t \to \infty$ of a super-critical 
 multi-type branching diffusion consists of a line density concentrated along a ray aligned with the left-eigenvector of the rate matrix $\gamma$.  
 \citet{burden2018mutation} have obtained a lowest order in the mutation rate solution which exhibits an asymptotic collapse onto such a 
 line density for a super-critical diffusion with $d = 2$ types and a symmetric rate matrix.  The following proposition generalises their calculation 
 to $d \ge 2$ types and an arbitrary rate matrix $\gamma$.  
 
 \begin{proposition}	\label{prop:leadingOrder}
 Consider a super-critical neutral multi-type branching diffusion with exponentially scaled generator Eq.~(\ref{genOfZ}) where $\alpha > 0$ and $\gamma$ 
 is parametrised by Eq.~(\ref{gammaParam}).  Then as $\theta \to 0$ the solution to the forward equation with initial condition $\mathbf{Z}(0) = \mathbf{x}(0)$ is 
\begin{equation}	\label{approxFZ}
f(t, \mathbf{z}; \mathbf{x}(0)) = \begin{cases}
\prod_{i = 1}^d f_{\text{1-allelle}}\left( t, z_i; x_i(0) \right) + \mathcal{O}(\theta), & t < t_{\rm c}, \\ \\
\det \left(e^{-\gamma(t - t_{\rm c})} \right) \prod_{i = 1}^d f_{\text{1-allelle}}\left( \infty, (\mathbf{z}^{\rm T} e^{-\gamma(t - t_{\rm c})})_i; x_i(0) \right) 
											+ \mathcal{O}(\theta), & t > t_{\rm c}, \\
\end{cases}
\end{equation}
where 
\begin{equation}	\label{tcDef}
t_{\rm c} = - \alpha^{-1} \log\theta.  
\end{equation} 
\end{proposition}

A formal proof of this proposition consisting of solving approximations to the characteristic equations arising from Eq.~(\ref{zetapde}) 
for $t < t_{\rm c}$ and $t > t_{\rm c}$ is given in \ref{sec:leadingOrderProof}.  The exact characteristic equations Eq.~(\ref{charEqns}) form a set 
of coupled ordinary differential equations, solution of which remains intractable in the general case.  The approximation relies 
on noting that for $\theta << 1$ the first, `genetic drift', term on the right hand side of Eq.~(\ref{genOfZ}) dominates for $t <  t_{\rm c}$, whereas the second, 
`mutation', term dominates for $t >  t_{\rm c}$.  Fig.~\ref{fig:3alleleCharacteristics_theta0.001} shows an example numerical solution to the 
exact characteristic equation and the small-$\theta$ approximate analytic solution for $d = 3$.  

\begin{figure}
 \centering
 \includegraphics[width=0.5\linewidth]{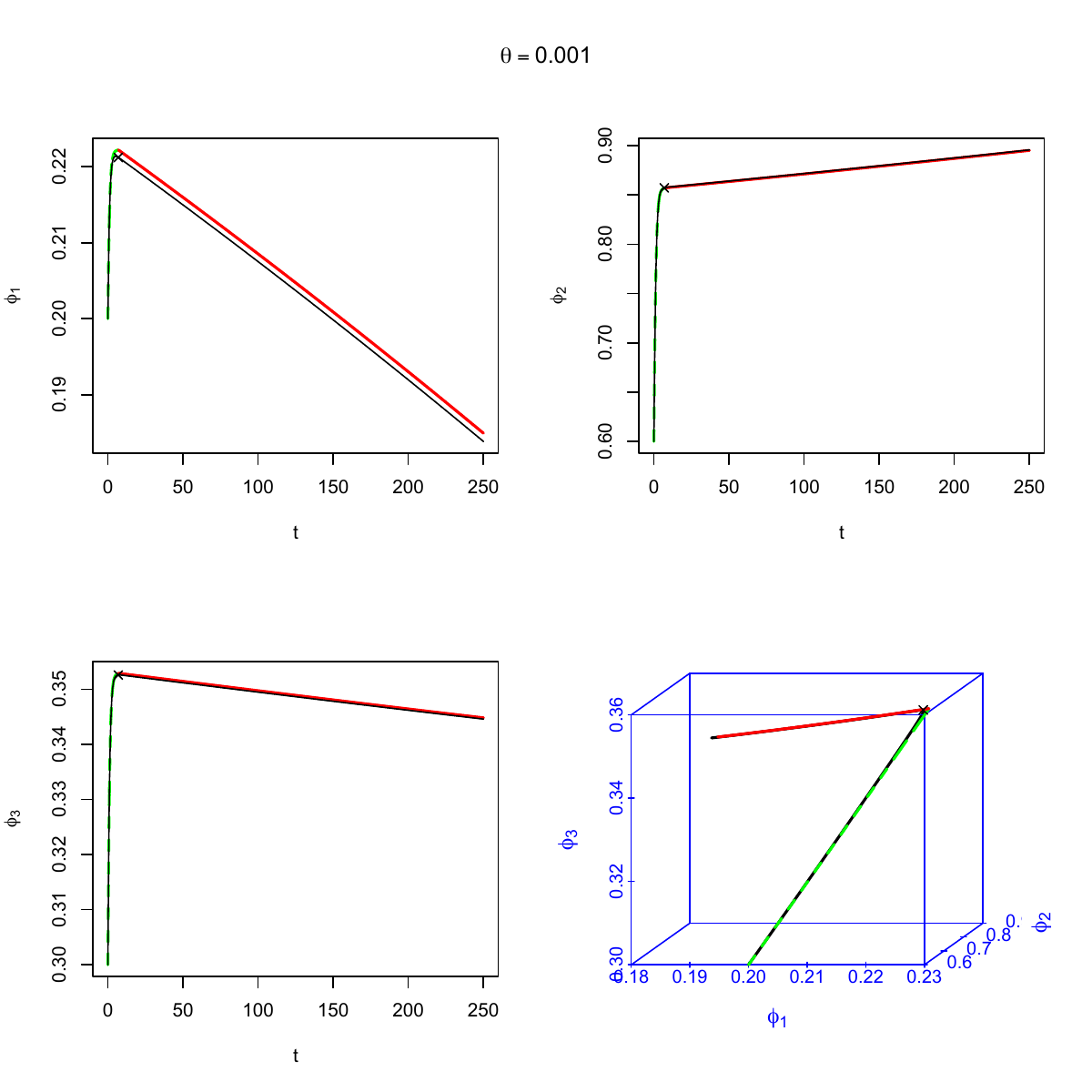}
  \caption{Numerical solution to the exact characteristic equations Eq.~(\ref{charEqns}) (black) and the approximate solution Eq.~(\ref{charEqnsApprox}) 
  (green dashed curve, $t < t_{\rm c}$ and red curve, $t > t_{\rm c}$).  Parameters are $d = 3$, $\alpha = 1$, $Q_{ij}$ randomly generated, 
  $\mathbf{x}(0) = (0.2, 0.6, 0.3)$, $\theta = 0.001$.  The point on the numerical solution to the exact equations corresponding corresponding to the time 
  $t_{\rm c}$ is marked with a cross.  }
 \label{fig:3alleleCharacteristics_theta0.001}
 \end{figure}

The first phase up to the critical changeover time $t_{\rm c}$ is the solution corresponding to the generator Eq.~(\ref{genOfZ}) with $\gamma \equiv 0$.  
It consists simply of $d$ independent Feller diffusions with initial conditions $Z_i(0) = x_i(0)$.  
Because $f_{\text{1-allelle}}(t, z; x(0))$ is only non-zero for $z \ge 0$, the support of the second phase of the approximate density Eq.~(\ref{approxFZ}) 
is the subregion of $\mathbb{R}_{\ge 0}^d$ bounded by the set of $(d - 1)$-dimensional hyperplanes 
$(\mathbf{z}^{\rm T} e^{-\gamma(t - t_{\rm c})})_i = 0$, $i = 1, \ldots, d$.  The $i$-th hyperplane is orthogonal to the $i$-th column of the 
matrix $e^{-\gamma(t - t_{\rm c})}$.  At $t = t_{\rm c}$ this is the hyperplane bounded by the axes orthogonal to the $i$-th axis.  For  $t > t_{\rm c}$
this hyperplane collapses inwards, approaching the stationary left eigenvector of $\gamma$ to become a line-density as $t \to \infty$.  

For comparison of data simulated from Eq.~(\ref{approxFZ}) using parameters  $\alpha = 0.75$, $d = 2$,  $\gamma_{12} = \gamma_{21} = 0.02$ with 
a multitype diffusion simulated by a multi-type BGW model, see \citet[Fig.~4(i) to (p)]{burden2018mutation}\footnote{To 
convert from the notation of \citet{burden2018mutation} to that of the current paper, replace $\kappa_0$ with $2\alpha$, $s$ with $\alpha t$ and 
$r_{ij}$ with $\gamma_{ij}$.}.  In these simulations it is clearly visible that the density undergoes a rapid changeover at time $t_{\rm c}$, after which 
it is, to a good approximation, sandwiched between two lines which collapse towards the left eigenvalue of the rate matrix.\\
 
 \noindent{\bf Example:} 
 A parent-independent rate matrix is one whose off-diagonal elements are of the form $\gamma_{ij} = \gamma_j$, independent of $i$.  In Eq.~(\ref{gammaParam}) set 
 $\theta = 2\sum_{j = 1}^d \gamma_j$ and $p_i = \gamma_i/\sum_{j = 1}^d \gamma_j$ for $i = 1, \ldots, d$. Then 
\[
Q_{ij} = p_j  - \delta_{ij}, \quad i, j = 1, \ldots, d,   
\]
where $p_j \ge 0$ and $\sum_{j = 1}^d p_j = 1$, and hence $Q^2 = -Q$, leading to   
\[
e^{\gamma t} = e^{\frac{1}{2}\theta t Q} = I + (1 - e^{-\frac{1}{2}\theta t}) Q.  
\]
In the second phase of Eq.~(\ref{approxFZ}), 
\[
\det \left(e^{-\gamma(t - t_{\rm c})} \right) = \exp{\rm tr}\left(-\gamma(t - t_{\rm c})\right) = e^{\frac{1}{2}(d - 1)\theta (t - t_{\rm c})},
\]
and 
\[
\left(\mathbf{z}^{\rm T} e^{-\gamma(t - t_{\rm c})}\right)_i = |\mathbf z|p_i + \left(z_i - |\mathbf z|p_i\right)  e^{\frac{1}{2}(d - 1)\theta (t - t_{\rm c})},
\]
where $|\mathbf z| = \sum_{l = 1}^d z_l$. 

%
%

\section{First order in $\theta$ corrections}
\label{sec:firstOrder}

In the spirit of various papers which have addressed the problem of approximating multi-type Wright-Fisher diffusions with small mutation rates 
\citep{roychoudhury2010sufficiency,vogl2014estimating,vogl2015inference,burden2016approximate,SchrempfHobolth17,VOGL2020106}, 
we seek to derive first order in $\theta$ corrections to the trivial case of $d$ independent Feller diffusions.  
In the case of Wright-Fisher diffusions, the first order approximation is equivalent to having at most one mutation in the coalescent tree of a sample 
back to the most recent common ancestor (MRCA) \citep{BurdenGriffiths19a,BurdenGriffiths19b}, and this observation is expected to carry over to 
multi-type branching diffusions.  However, in the previous section we have seen that, for super-critical branching diffusions, a second constraint manifests, 
namely that the time since the origin of the process should be less than the critical time $t_{\rm c}$ at which a rapid transition to a collapse onto a 
line density ocurs, as in Eq.~(\ref{approxFZ}).  For the remainder of the paper we will be assuming the unperturbed 
density to be that of $d$ independent Feller diffusions, 
\[
f^{(0)}(t, \mathbf{z}; \mathbf{x}(0)) = \prod_{i = 1}^d f_{\text{1-allelle}}\left( t, z_i; x_i(0) \right), \quad t, z_i \ge 0,  \quad i = 1, \ldots, d, 
\]
while noting that for the supercritical case $\alpha > 0$ the first-order perturbed solution only applies for 
$t < t_{\rm c}$.  

This particular restriction does not apply to the sub-critical and critical cases $\alpha \le 0$.  However, irrespective of the sign of $\alpha$, 
one expects the accuracy of a first order in $\theta$ approximation to be degraded as $\theta t$ increases to a point where more than one 
mutation can be expected in the coalescent tree of a sample.  Whereas the assumed unperturbed solution conditioned on survival maintains 
a memory of the initial conditions as $t \to \infty$, the mean of the sub-critical quasi-stationary distribution is known to align with the left stationary 
eigenvector of the rate matrix, and is therefore independent of initial allele abundances.  We comment further on how the perturbative 
small-$\theta$ approximation can be adapted to the quasi-stationary distribution in Section~\ref{sec:QSMoments}.  

%
%

\subsection{LT of the density of $\mathbf{Z}(t)$ to first order in $\theta$}

The LT in $\mathbf{z}$ of the exact density $f(t, \mathbf{z}; \mathbf{x}(0))$ satisfies Eq.~(\ref{zetapde})  with initial boundary condition 
\begin{equation}	\label{zetaIC}
\zeta(0; {\bm \phi}) = e^{-\bm{\phi}\cdot \mathbf{x}(0)}. 
\end{equation}
A complete solution remains intractable. The following lemma provides a perturbative solution to first order in $\theta$, which is obtained by integrating 
along perturbed characteristic curves.  The proof of Lemma~\ref{lemma:zetaPerturb} is given in \ref{appendix:zetaPerturb}.  
\begin{lemma}	\label{lemma:zetaPerturb}
The perturbative solution to Eq.~(\ref{zetapde}) with the above initial condition and $\gamma_{ij}$ as in Eq.~(\ref{gammaParam}) is 
\begin{eqnarray}	\label{zetaOfTPhi}
\zeta(t, \bm{\phi}) = \zeta^{(0)}(t, \bm{\phi}) \left(1 - \theta \sum_{j = 1}^d x_j(0) \xi_j^{(1)}(t, \bm{\phi}) + \mathcal{O}(\theta^2) \right), \quad \text{as }\theta \to 0, 
\end{eqnarray}
where 
\[
\zeta^{(0)}(t, \bm{\phi}) = \prod_{j = 1}^d  \zeta_{\text{1-allele}}(t, \phi_j; x_j(0)),  
\]
with $\zeta_{\text{1-allele}}$ defined in Eq.~(\ref{zeta1Allele}), 
\begin{eqnarray}	\label{xi1Of_tphi}
\xi_i^{(1)}(t, \bm{\phi}) 
	& = & \frac{\left(\frac{\phi_i}{2\alpha}\right)^2}{\left(1 + \frac{\phi_i}{\mu(t)}\right)^2} \sum_{j = 1}^d Q_{ij} \frac{1}{\frac{2\alpha}{\phi_j} - e^{-\alpha t}} \times \nonumber \\ 
	& & \qquad \left\{ \left(\frac{2\alpha}{\phi_i} - e^{-\alpha t} \right)^2 \alpha t - \left(  \frac{2\alpha}{\phi_i} 
							- \frac{2\alpha}{\phi_j} \right)^2 \log \left(1 + \frac{\phi_j}{\mu(t)}\right) \right\},  
\end{eqnarray}
and $\mu(t)$ defined by Eq.~(\ref{muBetaDef}).  
\end{lemma}
Inverting this LT is a formidable task.  We therefore make two simplifying assumptions.  First, suppose all individuals are initially of type $i$: 
\[
x_j(0) = x_0 \delta_{ij}, \quad j = 1, \ldots, d.  
\]
Then, 
\[
\zeta(t, \bm{\phi}) = e^{-x_0 \frac{\phi_i}{1 + \phi_i/\mu(t)}} \left(1 - \theta x_0 \xi_i^{(1)}(t, \bm{\phi}) + \mathcal{O}(\theta^2) \right).
\]

Second, suppose that the population survives to time $t$ and is descended from a single ancestral founder at time zero.  
Noting that the probability of extinction for the entire population given $\mathbf{Z}(0) = x_0 \mathbf{e}_i$ is the coefficient $e^{-x_0 \mu(t)}$ of the 
$\delta$-function in Eq.~(\ref{f1Allele}), conditioning on $\mathbf{Z}(t) > 0$ and a single ancestral founder gives
\begin{eqnarray}	\label{zeta1Ancestor}
\zeta^{\rm 1-ancestor}(t, \bm{\phi}) &  = & 
		\lim_{x_0 \to 0} \left(\text{LT of }\frac{f_{\mathbf{Z}(t)\mid\mathbf{Z}(0) = x_0 \mathbf{e}_i}(\mathbf{z}) - \delta(\mathbf{z}) e^{-x_0 \mu(t)}}
					{1 - e^{-x_0 \mu(t)}}\right)	\nonumber \\ 
	& = & \lim_{x_0 \to 0} \frac{\zeta(t, \bm{\phi}) - e^{-x_0 \mu(t)}} {1 - e^{-x_0 \mu(t)}}			\nonumber \\	
	& = & \lim_{x_0 \to 0} \frac{[1 - x_0 \frac{\phi_i}{1 + \phi_i/\mu(t)}][(1 - \theta x_0 \xi_i^{(1)}(t, \bm{\phi}) + \mathcal{O}(\theta^2)] - (1 - x_0\mu(t))} 
								{x_0\mu(t)}									\nonumber \\
	& = & - \frac{\phi_i/\mu(t)}{1 + \phi_i/\mu(t)} -  \frac{\theta \xi_i^{(1)}(t, \bm{\phi})}{\mu(t)} + 1  + \mathcal{O}(\theta^2)		\nonumber \\
	& = & \left(1 + \frac{\phi_i}{\mu(t)}\right)^{-1} - \frac{\theta \xi_i^{(1)}(t, \bm{\phi})}{\mu(t)} + \mathcal{O}(\theta^2).  
%
%
%
\end{eqnarray}
The first term is the LT of $f_\text{1-allele}^\text{1-ancestor}(t, z_i)\prod_{l \ne i} \delta(z_l)$, see Eq.~(\ref{1ancestor1allele}).   
It remains to find the inverse LT of Eq.~(\ref{xi1Of_tphi}).  
%
%

\subsection{Inverting the LT $\zeta^{\text{1-ancestor}}(t, \bm{\phi})$}

Inversion of the LT $\zeta^{\text{1-ancestor}}(t, \bm{\phi})$ requires inversion of $\xi_i^{(1)}(t, \bm{\phi})$ which is a lengthy calculation, details of which are given in \ref{appendix:invZeta}.  The result is 
summarised in the following theorem.  
\begin{theorem}	\label{thm:InverseZeta}
The inverse LT of $\zeta^{\text{1-ancestor}}(t, \bm{\phi})$ is 
\begin{equation}	\label{Z1ancestorDensity}
f^{\text{1-ancestor}}(t, \mathbf{z}) = \mu(t)^d F(\mu(t) \mathbf{z}), 
\end{equation}
where the function $F(\mathbf{u})$ for $\mathbf{u} \in \mathbb{R}_{\ge 0}^d$ takes the form of a sum over line and surface densities 
\begin{eqnarray}	\label{fU}
\lefteqn{F(\mathbf{u}) = F_0(u_i) \prod_{l \ne i} \delta(u_l) + \theta\left(F^{\text{line-par}}(u_i) \prod_{l \ne i} \delta(u_l) 
							+ \sum_{j \ne i} F_j^\text{line-perp}(u_j) \prod_{l \ne j} \delta(u_l) \right.}  \nonumber \\
	& & \qquad\qquad\qquad\qquad\qquad\qquad \left. + \sum_{j \ne i} F_j^\text{surface}(u_i, u_j) \prod_{l \ne i,j} \delta(u_l)\right) + \mathcal{O}(\theta^2),
\end{eqnarray}
as $\theta \to 0$, with
\begin{align*}
F_0(u_i) &= e^{-u_i}  - \beta(t) u_i e^{-u_i} \sum_{j \ne i} \frac{Q_{ij}}{\lambda_{ij}(u_i)}, \nonumber \\ 
F^{\text{line-par}}(u_i) & = \frac{1}{2\alpha} Q_{ii} \left\{ \alpha t (1 - u_i e^{\alpha t}) + (1 - e^{\alpha t})\gamma_{\rm EM} u_i \right\} e^{-u_i}, \nonumber \\
F_j^{\text{line-perp}}(u_j) &= \frac{1}{2\alpha} Q_{ij} \left\{E_1(u_j) - E_1(u_j e^{\alpha t}) e^{ (e^{\alpha t} - 1)u_j}\right\}, \nonumber \\ 
F_j^\text{surface}(u_i, u_j) &= \frac{1}{2\alpha} Q_{ij} \left\{(e^{\alpha t} - 1) u_j^{\lambda_{ij}(u_i)\theta -1} e^{-u_j} u_i \phantom{u_j^{\lambda_{ij}}} \right. \nonumber \\
					& \qquad\quad \left.\phantom{u_j^{\lambda_{ij}}} -e^{\alpha t}(u_i e^{\alpha t} - 2)E_1(u_j e^{\alpha t}) e^{(e^{\alpha t} - 1)u_j} 
										+ (u_i - 2)E_1(u_j) \right\}  e^{-u_i}, \nonumber 
\end{align*}  
$\beta(t)$ is defined by Eq.~(\ref{muBetaDef}), $\gamma_{\rm EM}$ is the Euler-Mascheroni constant, $E_1$ is the exponential integral 
\citep[Eq.~(5.1.1)]{Abramowitz:1965sf}, and  $\lambda_{ij}(u_i) > 0$ for $j \ne i$ are undetermined finite, analytic functions on $\mathbb{R}_{>0}$. 
\end{theorem}
The density $f^{\text{1-ancestor}}(t, \mathbf{z})$ requires some interpretation.  First note that it is conditioned on non-extinction of the population up to time $t$ 
and predicated on an assumption that the population is descended from a single ancestor of type $i$.  By approximating to first order in $\theta$ we are 
assuming at most one mutation has occurred in the ancestry of the final population.  
The line density $F_0(u_i) + \theta F^{\text{line-par}}(u_i)$ represents a population with no mutations in its ancestry and which is therefore non-segregating of type-$i$.  
The line density $\theta F_j^{\text{line-perp}}(u_j)$ represents a population whose ancestry contains exactly one mutation from type-$i$ to 
type-$j$ occurring before the MRCA of the population at time $t$, and which has therefore fixed throughout the surviving population.  The surface density 
$\theta F_j^\text{surface}(u_i, u_j)$ represents a population whose ancestry contains exactly one mutation from type-$i$ to type-$j$ occurring subsequent to 
the MRCA, which has therefore not fixed.  

The undetermined functions $\lambda_{ij}(u_i)$ arise when calculating the inverse LT of the part of Eq.~(\ref{xi1Of_tphi}) proportional to $\log(1 + \phi_j/\mu(t))$ 
with respect to $\phi_j$.  This inverse LT does not exist {\em per se}, but is regularised at Eq.~(\ref{LT_T3Part1})  within the proof by taking into account firstly that 
$\zeta^{\text{1-ancestor}}(t, \bm{\phi})$ only agrees with the exact LT of the density to first order in $\theta$, and secondly that known solutions to the 
forward Kolmogorov equation for diffusions with mutations  typically behave as $z^{\lambda\theta - 1}$ as the distance from the boundary $z$ tends to $0$, 
for some $\lambda > 0$.\footnote{Examples of this behaviour are the Wright-Fisher models with neutral parent-independent mutations, for which the stationary 
distribution is Dirichlet~\citep[Eq.~(5.126)]{Ewens:2004kx}, and the first order approximate stationary solution with general 
mutations~\citep[Appendix~A]{burden2016approximate}.}  Though it may be possible to calculate the $\lambda_{ij}(u_i)$, in many applications it is the 
moments of the density which are of interest.  
In the next section we find that the moments to first order in $\theta$ of the population $\mathbf{X}(t)$ conditioned on a single ancestor are independent 
of $\lambda_{ij}(u_i)$.  This is to be expected as these moments should be determined entirely by the LT $\zeta^{\text{1-ancestor}}(t, \bm{\phi})$.  

%
%

\section{Moments of $\mathbf{X}(t)$}
\label{sec:XMoments}

Introduce the notation $\mathbf{X}^\mathbf{n} = \prod_{l = 1}^d X_l^{n_l}$ and $|n| = \sum_{l = 1}^d n_l$.  From Eq.~(\ref{ZDef}) the moments of 
$\mathbf{X}(t)$ conditioned on a single ancestral founder at time 0 are 
\begin{equation}	\label{momentsUtoX}
\mathbb{E}[\mathbf{X}(t)^\mathbf{n} \mid \mathcal{E}_i] = e^{|n|\alpha t}\mathbb{E}[\mathbf{Z}(t)^\mathbf{n} \mid \mathcal{E}_i]  
	= \beta(t)^{|n|}\mathbb{E}[\mathbf{U}(t)^\mathbf{n} \mid \mathcal{E}_i],
\end{equation}
where $\mathcal{E}_i$ is the event that the the population survives to time $t$ from a founding ancestor of type $i$, and we define 
$\mathbf{U}(t) := \mu(t)\mathbf{Z}(t)$.  From Eq.~(\ref{Z1ancestorDensity})
\[
\mathbb{E}[\mathbf{U}(t)^\mathbf{n} \mid \mathcal{E}_i] = \int_{(\mathbb{R}_{\ge 0})^d} \mathbf{u}^\mathbf{n} F(\mathbf{u}) d^d \mathbf{u},  
\]
with $F(\mathbf{u})$ as in Eq.~(\ref{fU}). 

Before calculating moments, we confirm the normalisation of $f^{\text{1-ancestor}}(t, \mathbf{z})$ to order $\theta$ by setting $\mathbf{n} = \mathbf{0}$.  
Thus  
\begin{align*}
\int_{(\mathbb{R}_{\ge 0})^d} f^{\text{1-ancestor}}(t, \mathbf{z}) d\mathbf{z} &=  \int_{(\mathbb{R}_{\ge 0})^d} F(\mathbf{u}) d\mathbf{u} \nonumber \\
	&= \int_0^\infty F_0(u_i)du_i + \theta\left(\int_0^\infty F^{\text{line-par}}(u_i) du_i  \right. \nonumber \\
	& \quad 	\left. + \sum_{j \ne i}\int_0^\infty F_j^{\text{line-perp}}(u_j) du_j
						+ \sum_{j \ne i} \int_0^\infty \int_0^\infty F_j^{\text{surface}}(u_i, u_j) du_i du_j \right).  
\end{align*}
Then make use of the integrals in \ref{appendix:MomentIntegrals} to get,
\begin{align*}
\int_0^\infty F_0(u_i)du_i &= 1 - \beta(t) \sum_{j \ne i} Q_{ij} \int_0^\infty \frac{u_i e^{-u_i}}{\lambda_{ij}(u_i)} du_i, 		\\
\theta \int_0^\infty F^{\text{line-par}}(u_i) du_i & = \frac{\theta}{2\alpha} Q_{ii} (1 - e^{\alpha t})(\alpha t + \gamma_{\rm EM}), 	\\
\theta \sum_{j \ne i}\int_0^\infty F_j^{\text{line-perp}}(u_j) du_j & = \frac{\theta}{2\alpha} Q_{ii}  \left(\frac{1}{e^{\alpha t} - 1} \alpha t - 1 \right),	\\
\theta \sum_{j \ne i} \int_0^\infty \int_0^\infty F_j^{\text{surface}}(u_i, u_j) du_i du_j& = \beta(t) \sum_{j \ne i} Q_{ij} \int_0^\infty \frac{u_i e^{-u_i}}{\lambda_{ij}(u_i)} du_i 
					+ \theta Q_{ii} \frac{e^{\alpha t} - 1}{2\alpha} \gamma_{\rm EM} \\
				& \qquad\qquad  - \frac{\theta}{2\alpha} Q_{ii} \left( \frac{e^{\alpha t}(2 - e^{\alpha t})}{e^{\alpha t} - 1} \alpha t - 1 \right) + \mathcal{O}(\theta^2).  
\end{align*}
Summing these gives $1 + \mathcal{O}(\theta^2)$, as required.  

For the moments $\mathbb{E}[\mathbf{U}(t)^\mathbf{n} \mid \mathcal{E}_i]$ with $|n| \ge 1$, the $\delta$-functions in Eq.~(\ref{fU}) ensure that the only 
non-zero moments to order $\theta$ are 
\begin{align*}
\mathbb{E}[U_i(t)^m \mid \mathcal{E}_i] 
	&= \int_0^\infty u_i^m F_0(u_i)du_i + \theta \int_0^\infty u_i^m F^{\text{line-par}}(u_i) du_i \\
	& \qquad\qquad + \theta \sum_{j \ne i} \int_0^\infty\int_0^\infty  u_i^m F^\text{surface}(u_i, u_j) du_i du_j, \\
\mathbb{E}[U_j(t)^n \mid \mathcal{E}_i] 
	&= \theta \int_0^\infty u_j^n F^{\text{line-perp}}(u_j) du_j + \theta \int_0^\infty\int_0^\infty  u_j^n F^\text{surface}(u_i, u_j) du_i du_j, \quad j \ne i, \\
\mathbb{E}[U_i(t)^m, U_j(t)^n \mid \mathcal{E}_i] 
	&= \theta \int_0^\infty \int_0^\infty u_i^m u_j^n F^\text{surface}(u_i, u_j) du_i du_j, \quad j \ne i,  
\end{align*}
where $m, n \ge 1$.  By straightforward application of the integrals in \ref{appendix:MomentIntegrals}, the required moments are 
\begin{align}
\mathbb{E}[U_i(t)^m \mid \mathcal{E}_i] &= m!\left[1 + \frac{\theta}{2\alpha}Q_{ii} \left\{ \frac{1 - me^{\alpha t}}{1 - e^{\alpha t}}  \alpha t - m + 1 \right\} \right] , 
											\label{EXim} \\
\mathbb{E}[U_j(t)^n \mid \mathcal{E}_i] &= \frac{\theta}{2\alpha} Q_{ij} (e^{\alpha t} - 1)(n - 1)! \times  \nonumber \\
						&\qquad \left\{1 - \frac{n}{(1 - e^{\alpha t})^n} \left( \alpha t + \sum_{l = 1}^n \frac{(1 - e^{\alpha t})^l}{l} \right) \right\},  
											\label{EXjn} \\
\mathbb{E}[U_i(t)^m U_j(t)^n \mid \mathcal{E}_i] 	&=  \frac{\theta}{2\alpha} Q_{ij} (n - 1)!m! \left\{ (m + 1)e^{\alpha t} - \frac{m + 2n + 1}{n + 1} \right.  \nonumber \\
						& \quad \left. - \, (2 - (m + 1)e^{\alpha t}) \frac{ne^{\alpha t}}{(1 - e^{\alpha t})^{n + 1}} 
							\left( \alpha t + \sum_{l = 1}^n \frac{(1 - e^{\alpha t})^l}{l} \right) 	\right\},  \label{EXimXjn}
\end{align}
for $m, n \ge 1$ and $j \ne i$.  Note that all contributions in terms of the undetermined functions $\lambda_{ij}(u_i)$ cancel.  
Alternative forms of Eqs.~(\ref{EXim}) and (\ref{EXjn}) in terms of ordinary hypergeometric functions are 
 \begin{align*}
\mathbb{E}[U_j(t)^n \mid \mathcal{E}_i] &= \frac{\theta}{2\alpha} Q_{ij} (e^{\alpha t} - 1)(n - 1)! \times  \\
						&\qquad \left\{1 - \frac{n}{n + 1} (e^{\alpha t} - 1) \,_2F_1(1, n+ 1; n + 2; 1 - e^{\alpha t}) \right\}, \\
\mathbb{E}[U_i(t)^m U_j(t)^n \mid \mathcal{E}_i] 	&=  \frac{\theta}{2\alpha} Q_{ij} \frac{(n - 1)!m!}{n + 1} \left\{(m + 1)(n + 1)e^{\alpha t} - m - 2n - 1 \right.   \\
						&\qquad\qquad \left. + \, (2 - (m + 1)e^{\alpha t}) ne^{\alpha t}  \,_2F_1(1, n+ 1; n + 2; 1 - e^{\alpha t}) \right\}. 
\end{align*}
 
 Moments for the critical case are found by taking the limit $\alpha \to 0$:  
 \begin{align*}
\lim_{\alpha \to 0} \mathbb{E}[U_i(t)^m \mid \mathcal{E}_i] &=  m! \left[ 1 + \tfrac{1}{4} \theta Q_{ii} (m + 1) t \right],  \\
\lim_{\alpha \to 0} \mathbb{E}[U_j(t)^n \mid \mathcal{E}_i] &=  \tfrac{1}{2} \theta Q_{ij} (n - 1)! t, \\
\lim_{\alpha \to 0} \mathbb{E}[U_i(t)^m U_j(t)^n \mid \mathcal{E}_i] 	&=  \theta Q_{ij} m!(n - 1)! \frac{m + n + 1}{(n + 1)(n + 2)} t,  
\end{align*}
for $m, n \ge 1$ and $j \ne i$.  

Finally, moments of $\mathbf{X}$ are found from moments of $\mathbf{U}$ by multiplying by an appropriate power of $\beta(t)$ 
according to Eq.~(\ref{momentsUtoX}).  

%

\subsection{Consistency with sub-critical quasi-stationary moments}
\label{sec:QSMoments}

For the sub-critical case $\alpha < 0$ the above results should be consistent with first order in $\theta$ quasi-stationary moments derived in 
\citet[Section~6.1]{BurdenGriffiths23}.  However, those moments were derived by perturbing the density 
\[
2|\alpha|\sum_{i = 1}^d \pi_i e^{-2|\alpha|x_i}, \quad x_i \ge 0, 
\]
which is the $\theta \to 0$ limit of the exact quasi-stationary distribution whose mean is known to be aligned with the left-eigenvector 
$\bm{\pi}$ of the rate matrix $Q$ \citep[Lemma~2 and Theorem~1]{BurdenGriffiths23}.  Here $\bm{\pi}$ satisfies $\sum_{i = 1}^d \pi_i = 1$ 
and $\sum_{i = 1}^d \pi_i Q_{ij} = 0$.  

By contrast, the above first order in $\theta$ calculation of moments implicitly assumes only at most one mutation has occurred 
in the coalescent tree since founding by a single ancestor of a specified type at time zero, and therefore the distribution maintains 
a memory of the allele-type of that ancestral founder.  Note that, in the quasi-stationary limit, the MRCA of the surviving 
population must exist since, as argued in \citet[Eq.~(11)]{BurdenGriffiths23}, only the $l = 1$ term in the density 
Eq.~(\ref{FellersPoissonGammaMixture}) survives in this limit. In order to establish the consistency of the current paper 
with \citet{BurdenGriffiths23}, the type of this ancestral individual is set in a way consistent with the above unperturbed 
density, thus allowing up to one more perturbation to contribute to the perturbed quasi-stationary density.  
That is, we set an initial founder randomly chosen to be of type-$i$ with probability $\pi_i \ge 0$, leading to the quasi-stationary  moments 
\begin{equation}	\label{EofXforQS}
\mathbb{E}_{\rm QS}[\mathbf{X}^\mathbf{n}] = \lim_{t \to \infty}\sum_{i = 1}^d \pi_i \mathbb{E}[\mathbf{X}(t)^\mathbf{n} \mid \mathcal{E}_i]
			= \lim_{t \to \infty}\beta(t)^{|n|} \sum_{i = 1}^d \pi_i \mathbb{E}[\mathbf{U}(t)^\mathbf{n} \mid \mathcal{E}_i],  
\end{equation} 
with the conditioned expectations $\mathbb{E}[\mathbf{X}(t)^\mathbf{n} \mid \mathcal{E}_i]$ calculated in Eqs.~(\ref{EXim}) to (\ref{EXimXjn}).  
Details of the calculation showing that this expectation does indeed agree with \citet{BurdenGriffiths23} are given in \ref{appendix:QSmoments}.

%

\section{Conclusions}
\label{sec:Conclusions}

We have determined approximate solutions to first order in the small overall mutation rate $\theta$ for the multi-type neutral branching 
diffusion governed by the generator Eq.~(\ref{genOfXMultitype}).  This diffusion is a generalisation of the Feller diffusion to multiple types 
and can be thought of as an analogue of the diffusion limit of neutral Wright-Fisher or Moran diffusions in which population size is allowed 
to vary stochastically.  

In Section~\ref{sec:leadingOrder} the zeroth order approximation in $\theta$ to the evolution of the multi-type population density is determined.  
The interesting case (see Proposition~\ref{prop:leadingOrder}) is the supercritical case in which we find that the asymptotic density at small 
$\theta$ is not closely approximated by the simple the naive solution corresponding to independent Feller diffusions for all $t$, but instead undergoes a 
rapid change in behaviour at a critical time $t_{\rm c}$, after which the support of the density is sandwiched between a set of hyperplanes 
which collapse towards the stationary left eigenvector of the mutation rate matrix.  This is consistent with the known stationary solution of both 
discrete and continuous-state branching diffusions.  The proof of Proposition~\ref{prop:leadingOrder} is a generalisation to arbitrary rate matrices 
of an earlier result by \citet{burden2018mutation} which was restricted to $d = 2$ types and a symmetric rate matrix.  In that paper numerical 
simulations were carried out with illustrated the rapid change from a perturbation of independent Feller diffusions to a collapse towards a line density.  

In Section~\ref{sec:firstOrder} the first order in $\theta$ corrections to the the naive zeroth order density corresponding to independent Feller diffusions are 
calculated, noting that such corrections are only valid in the supercritical case for $t < t_{\rm c}$, but are valid more generally for the critical and 
sub-critical cases.  Our approach is to solve the forward Kolmogorov equation by Laplace transforming with respect to the population size, and then 
solving the resulting first order partial differential equation Eq.~(\ref{zetapde}) by integrating along characteristic curves perturbed from 
those of the $\theta = 0$ solution.  The solution (see Theorem~\ref{thm:InverseZeta}) is represented in terms of line and surface densities on the boundary of 
the domain $(\mathbb{R}_{\ge 0})^d$ of the density.  The need for such a representation is inevitable for a first order approximation, 
firstly because the $\mathcal{O}(\theta)$ approximation implicitly allows at most one mutation 
in the coalescent tree of any sample of the population, and secondly because of the singular nature of the unknown exact density near the boundary.  
A similar representation was found necessary by \citet{BurdenGriffiths23} who calculated the first-order corrections to the quasi-stationary distribution 
for a subcritical multi-type diffusion.  

The representation of the density in terms of line and surface densities lends itself to straightforward calculation to first order in $\theta$ of moments 
of $\mathbf{X}(t)$, which are calculated in Section~\ref{sec:XMoments}.  Ideally one would like to calculate sampling distributions, which require moments 
conditioned on total population size, so adding an extra level of difficulty to the calculation.  A promising alternative to solving the forward Kolmogorov 
equation, which has proved useful for the Wright-Fisher diffusion~\citep{BurdenGriffiths19a,BurdenGriffiths19b}, is to work directly from the coalescent tree.  
In such an approach we hope to exploit the known distribution of coalescent times for Feller diffusions~\citep{BurdenGriffiths24,BurdenGriffiths25}.  

%

\section*{Acknowledgement}

CB would like to thank Robert Griffiths for helpful discussions and a careful reading of the manuscript.

%
%

\appendix

%

\section{Proof of Proposition~\ref{prop:leadingOrder}}
\label{sec:leadingOrderProof}

For a general rate matrix $(\gamma_{ij})_{i, j = 1}^d$ the LT of $f(t, \mathbf{z})$ evolves via Eq.~(\ref{zetapde}) which rearranges to 
\[
\frac{\partial \zeta}{\partial t} + 
	\sum_{j = 1}^d \left( \tfrac{1}{2} e^{-\alpha t} \phi_j^2 - \sum_{i = 1}^d \gamma_{ji} \phi_i \right) \frac{\partial \zeta}{\partial \phi_j} =0.  
\]
The exact equations for the characteristic curves are 
\begin{equation}	\label{charEqns}
\frac{d\phi_j}{dt} = \tfrac{1}{2} e^{-\alpha t} \phi_j^2 - \sum_{i = 1}^d \gamma_{ji} \phi_i , \quad j = 1, \ldots, d.
\end{equation}
Setting an initial condition $\bm{\phi} = \bm{\xi}$ at $t = 0$ and solving for $\bm{\xi}$ in terms of any point $(t, \bm{\phi})$ by tracing the solution to the characteristic 
curve back to its origin at time zero, one obtains the solution for the LT from 
\begin{equation}	\label{zetaSoln}
\zeta(t, \bm\phi; \mathbf{x}(0)) = \mathbb{E}\left[e^{-\bm{\xi}\cdot \mathbf{Z}(0)}\right] =  \exp \{-\bm{\xi}\cdot \mathbf{x}(0)\}.  
\end{equation}
 
In terms of a rescaled time $\tau = t/t_{\rm c}$, where $t_{\rm c}$ is defined in Eq.~(\ref{tcDef}), the characteristic equations become 
\begin{eqnarray*}
\frac{d\phi_j}{d\tau} & = & \tfrac{1}{2} t_{\rm c} e^{-\alpha t_{\rm c} \tau} \phi_j^2 - \tfrac{1}{2} t_{\rm c} \theta \sum_{i = 1}^d Q_{ji} \phi_i \nonumber \\
	& = & -(2\alpha)^{-1} \log\theta \{\theta^\tau \phi_j^2 + \theta \sum_{i = 1}^d Q_{ji} \phi_i\}.  
\end{eqnarray*}
If $0 < \theta << 1$ and $\alpha > 0$ the first term on the right hand side dominates for $t < t_{\rm c}$ and the second term dominates for $t > t_{\rm c}$.  
Accordingly, approximate the characteristic equations by 
\begin{equation}	\label{charEqnsApprox}
\frac{d\phi_j}{dt} \approx  \begin{cases}
		\displaystyle \tfrac{1}{2} e^{-\alpha t} \phi_j^2		& t < t_{\rm c}; \\
		- \displaystyle \sum_{i = 1}^d \gamma_{ji} \phi_i 	& t > t_{\rm c}, 
		\end{cases}
\end{equation}
for $j = 1, \ldots, d$, where it is assumed that $\theta << 1$.  

The solution to the approximate characteristic equation for $t < t_{\rm c}$ is 
\[
\phi_i = \frac{\xi_i}{1 - \frac{\xi_i}{2\alpha}(1 - e^{-\alpha t})} =  \frac{\xi_i}{1 - \xi_i/\mu(t)}, \qquad i = 1, \ldots, d, \quad t < t_{\rm c},
\]
which inverts to 
\[
\xi_i = \frac{\phi_i}{1 + \phi_i/\mu(t)} , \qquad i = 1, \ldots, d, \quad t < t_{\rm c}.  
\]
The approximate LT for $t < t_{\rm c}$ is then 
\begin{eqnarray*}
\zeta(t, \bm\phi; \mathbf{x}(0)) & = & \exp \left\{- \sum_{i = 1}^d  x_i(0) \frac{\phi_i}{1 + \phi_i/\mu(t)} \right\} \nonumber \\
					& = & \prod_{i = 1}^d \zeta_{\text{1-allele}}(t, \phi_i; x_i(0)),   \qquad t < t_{\rm c}, 
\end{eqnarray*}
with $\zeta_{\text{1-allelle}}$ defined by Eq.~(\ref{zeta1Allele}).  The inverse LT then gives the leading order contribution to the first line of Eq.~(\ref{approxFZ}).  

The boundary condition for the characteristic curves at the beginning of the second phase, $t > t_{\rm c}$, is 
\[
c_i :=  \phi_i(t_{\rm c}) \approx \frac{\xi_i}{1 - \frac{\xi_i}{2\alpha}}, \qquad i = 1, \ldots, d, 
\]
where $\mathcal{O}(\theta)$ terms have been ignored in the approximation.  This inverts to 
\begin{equation}	\label{xiInTermsOfC}
\xi_i \approx \frac{c_i}{1 + \frac{c_i}{2\alpha}}, \qquad i = 1, \ldots, d.   
\end{equation}
The solution to the approximate characteristic equation Eq.~(\ref{charEqnsApprox}) with this boundary condition is 
\[
\bm{\phi} = e^{-\gamma(t - t_{\rm c})}  \mathbf{c} = \pi(t)^{-1} \mathbf{c}, \qquad t > t_{\rm c},  
\]
where $\bm{\phi} = (\phi_1, \ldots, \phi_d)^{\rm T}$, $\mathbf{c} = (c_1, \ldots, c_d)^{\rm T}$, and 
$\pi(t) = e^{\gamma(t - t_{\rm c})}$ is a $d \times d$ matrix.  Reinserting Eq.~(\ref{xiInTermsOfC}) with $\mathbf{c} = \pi(t) \bm{\phi}$ into 
Eq.~(\ref{zetaSoln}) gives the approximate LT for $t > t_{\rm c}$, 
\begin{eqnarray}	\label{zetaPhase2}
\zeta(t, \bm\phi; \mathbf{x}(0)) & = & \exp \left\{- \sum_{i = 1}^d  x_i(0) \frac{c_i}{1 + \frac{c_i}{2\alpha}} \right\} \nonumber \\
					& = & \prod_{i = 1}^d \zeta_{\text{1-allelle}}\left(\infty, (\pi(t) \bm{\phi})_i; x_i(0)\right), \qquad t > t_{\rm c}.
\end{eqnarray}

It remains to confirm that Eq.~(\ref{zetaPhase2}) is the LT of the second line of Eq.~(\ref{approxFZ}).  Taking the LT and using the change 
of variables $\tilde{\mathbf{z}}^{\rm T} = \mathbf{z}^{\rm T}\pi(t)^{-1}$, we have
\begin{eqnarray*}
\text{LT of Eq.~(\ref{approxFZ})}_{t > t_{\rm c}} & = & \det \pi(t)^{-1} \int_{\mathbb{R}^d} e^{-\mathbf{z}^{\rm T} \bm{\phi}} \prod_{i = 1}^d f_{\text{1-allelle}}\left( \infty, (\mathbf{z}^{\rm T} \pi(t)^{-1})_i; x_i(0) \right) d^d\mathbf{z} \\
& = & \int_{\mathbb{R}^d} e^{-\tilde{\mathbf{z}}^{\rm T} \pi(t) \bm{\phi}} \prod_{i = 1}^d f_{\text{1-allelle}}\left( \infty, \tilde{z}_i; x_i(0) \right) d^d\tilde{\mathbf{z}} \\
& = & \prod_{i = 1}^d \int_\mathbb{R} e^{-\tilde{z}_i (\pi(t) \bm{\phi})_i} f_{\text{1-allelle}}\left( \infty, \tilde{z}_i; x_i(0) \right) d\tilde{z}_i \\
& = &  \prod_{i = 1}^d \zeta_{\text{1-allelle}}\left(\infty, (\pi(t) \bm{\phi})_i; x_i(0)\right), 
\end{eqnarray*}
as required.  \qed

%
%

\section{Proof of Lemma~\ref{lemma:zetaPerturb}: Perturbative solution of the LT in $\mathbf{Z}$ to first order in $\theta$}
\label{appendix:zetaPerturb}

Returning to the exact characteristic equation, Eq.~(\ref{charEqns}), with an initial condition $\bm{\phi} = \bm{\xi}$ at $t = 0$, set 
\[
\phi_i(t, \bm{\xi}) = \phi_i^{(0)}(t, \bm{\xi}) + \theta \phi_i^{(1)}(t, \bm{\xi}) +o(\theta),  \qquad i = 1, \ldots, d, 
\]
as $\theta \to 0$, where, for each $i$,  
\begin{equation}	\label{phi0Defn}
\phi_i^{(0)}(t, \bm{\xi}) =  \frac{\xi_i}{1 - \frac{\xi_i}{2\alpha}(1 - e^{-\alpha t})} , \qquad i = 1, \ldots, d
\end{equation}
is the unperturbed solution to the characteristic equation, Eq.~(\ref{charEq1Allele}).  
Substituting into Eq.~(\ref{charEqns}) with $\gamma_{ij}$ as in Eq.~(\ref{gammaParam}) and keeping terms to $\mathcal{O}(\theta)$, we have 
\begin{eqnarray*}
\frac{d\phi_i^{(1)}(t)}{dt} & = & \phi_i^{(0)}(t) \phi_i^{(1)}(t) e^{-\alpha t} - \frac{1}{2} \sum_{j = 1}^d Q_{ij} \phi_j^{(0)}(t) \nonumber \\
    & = & \frac{\xi_i e^{-\alpha t}}{1 - \frac{\xi_i}{2\alpha}(1 - e^{-\alpha t})} \phi_i^{(1)}(t) - 
    						\frac{1}{2} \sum_{j = 1}^d Q_{ij} \frac{\xi_j}{1 - \frac{\xi_j}{2\alpha}(1 - e^{-\alpha t})}, \quad i = 1, \ldots, d.  
 \end{eqnarray*}

For each $i$ this is an equation of the form 
\[
\frac{d\phi^{(1)}}{dt} - A(t) \phi^{(1)} = - B(t).   
\]
Given the initial condition $\phi^{(1)}(0) = 0$, the solution is
\[
\phi^{(1)}(t) = - \int_0^t B(\tau) e^{\int_\tau^t A(\eta) d\eta} d\tau.  
\]
Calculating the inner integral, 
\begin{eqnarray*}
e^{\int_\tau^t A(\eta) d\eta} & = & \exp \int_\tau^t \frac{\xi e^{-\alpha \eta} d\eta}{1 - \frac{\xi}{2\alpha}(1 - e^{-\alpha \eta})} \nonumber \\
	& = & \exp \left[ -2\log \left\{1 - \frac{\xi}{2\alpha}(1 - e^{-\alpha\eta})\right\}\right]_{\eta = \tau}^t  \nonumber \\
	& = & \left(\frac{1 - \frac{\xi}{2\alpha}(1 - e^{-\alpha\tau})}{1 - \frac{\xi}{2\alpha}(1 - e^{-\alpha t})}\right)^2.  
\end{eqnarray*}
Assembling the pieces then gives, 
\begin{eqnarray}	\label{phi1Soln}
\phi_i^{(1)}(t, \bm{\xi}) & = & - \frac{1}{2} \sum_{j = 1}^d Q_{ij} \int_0^t \frac{\xi_j}{1 - \frac{\xi_j}{2\alpha}(1 - e^{-\alpha \tau})} 
				\left(\frac{1 - \frac{\xi_i}{2\alpha}(1 - e^{-\alpha\tau})}{1 - \frac{\xi_i}{2\alpha}(1 - e^{-\alpha t})}\right)^2 d\tau \nonumber \\
	& = & - \frac{1}{2} \sum_{j = 1}^d Q_{ij} \int_0^t \phi_j^{(0)}(\tau, \bm{\xi}) \left(\frac{\phi_i^{(0)}(t, \bm{\xi})}{\phi_i^{(0)}(\tau, \bm{\xi})}\right)^2 d\tau.  
\end{eqnarray}

Given any $(t,\bm{\phi})$ with $t > 0$ and $\bm{\phi} = (\phi_1, \ldots, \phi_d)$ we trace the characteristic curve back to $(0, \bm{\xi})$ by 
solving 
\begin{equation}	\label{bmXiEqn}
\phi_i = \phi_i^{(0)}(t, \bm{\xi}) + \theta \phi_i^{(1)}(t, \bm{\xi}), \qquad i = 1, \ldots, d, 
\end{equation}
for $\bm{\xi}(t, \bm{\phi})$, where $\bm{\phi}^{(0)}(t, \bm{\xi})$ and $\bm{\phi}^{(1)}(t, \bm{\xi})$ are given by Eqs.~(\ref{phi0Defn}) and (\ref{phi1Soln}) respectively.  
Then the required LT consistent with the initial condition Eq.~(\ref{zetaIC}) will then be  
\begin{equation}	\label{LTofDensity}
\zeta(t, \bm{\phi}) = \exp\left(-\mathbf{x}(0)\cdot \bm{\xi}(t, \bm{\phi})\right).  
\end{equation}

Note that we only need to solve Eq.(\ref{bmXiEqn}) to first order in $\theta$.  Accordingly, set 
\[
\bm{\xi}(t, \bm{\phi}) = \bm{\xi}^{(0)}(t, \bm{\phi}) + \theta \bm{\xi}^{(1)}(t, \bm{\phi}) + o(\theta), 
\]
where 
\begin{equation}	\label{xi0Of_tphi}
\xi_i^{(0)}(t, \bm{\phi}) = \frac{\phi_i}{1 + \frac{\phi_i}{2\alpha}(1 - e^{-\alpha t})} = \frac{\phi_i}{1 + \phi_i/\mu(t)}, \qquad i = 1, \ldots, d, 
\end{equation}
is the inverse of Eq.~(\ref{phi0Defn}).  The first term in Eq.~(\ref{bmXiEqn}) is, to $\mathcal{O}(\theta)$,  
\begin{eqnarray}	\label{bmXiEqn1stTerm}
\phi_i^{(0)}(t, \bm{\xi}(t, \bm{\phi})) & = & \phi_i^{(0)}\left(t, \bm{\xi}^{(0)}(t, \bm{\phi}) + \theta\bm{\xi}^{(1)}(t, \bm{\phi})\right) \nonumber \\
	& = & \phi_i^{(0)}(t, \bm{\xi}^{(0)}(t, \bm{\phi})) + \theta \sum_{j = 1}^d \xi_j^{(1)}(t, \bm{\phi}) \partial_j \phi_i^{(0)}(t, \bm{\xi}^{(0)}(t, \bm{\phi})) \nonumber \\  
	& = & \phi_i + \theta \xi_i^{(1)}(t, \bm{\phi}) \left(1 + \frac{\phi_i}{2\alpha}(1 - e^{-\alpha t})\right)^2,  
\end{eqnarray}
where we have used 
\[
\partial_j \phi_i^{(0)}(t, \bm{\xi}) := \frac{\partial \phi_i^{(0)} (t, \bm{\xi})}{\partial \xi_j}  = \frac{\delta_{ij}}{\left(1 - \frac{\xi_i}{2\alpha}(1 - e^{-\alpha t})\right)^2}.  
\]
The second term in Eq.~(\ref{bmXiEqn}) is, to $\mathcal{O}(\theta)$,  
\begin{eqnarray}	\label{bmXiEqn2ndTerm}
\theta \phi_i^{(1)}(t, \bm{\xi}^{(0)}(t, \bm{\phi})) & = & 
	- \tfrac{1}{2}\theta \sum_{j = 1}^d Q_{ij} \int_0^t \phi_j^{(0)}(\tau, \bm{\xi}^{(0)}(t, \phi)) 
						\left(\frac{\phi_i^{(0)}(t, \bm{\xi}^{(0)}(t, \phi))}{\phi_i^{(0)}(\tau, \bm{\xi}^{(0)}(t, \phi))}\right)^2 d\tau \nonumber \\
	& = & - \tfrac{1}{2}\theta \sum_{j = 1}^d Q_{ij} \phi_j \int_0^t \frac{\left(1 + \frac{\phi_i}{2\alpha} (e^{-\alpha\tau} - e^{-\alpha t})\right)^2}
												{1 + \frac{\phi_j}{2\alpha} (e^{-\alpha\tau} - e^{-\alpha t})} d\tau,  
\end{eqnarray}
where we have used that 
\begin{eqnarray*}
\phi_i^{(0)}(\tau, \bm{\xi}^{(0)}(t, \phi)) & = & \frac{\xi_i^{(0)}(t, \phi)}{1 - \frac{\xi_i^{(0)}(t, \phi)}{2\alpha}(1 - e^{-\alpha\tau})} \nonumber \\
	& = & \frac{\phi_i / (1 + \frac{\phi_i}{2\alpha}(1 - e^{-\alpha t}))}{1 - \{(\phi_i/(2\alpha)) / (1 + \frac{\phi_i}{2\alpha}(1 - e^{-\alpha t})\}(1 - e^{-\alpha\tau})}  \nonumber \\
	& = & \frac{\phi_i}{1 + \frac{\phi_i}{2\alpha}(e^{-\alpha\tau} - e^{-\alpha t})}.  
\end{eqnarray*}

Plugging Eqs.~(\ref{bmXiEqn1stTerm}) and (\ref{bmXiEqn2ndTerm}) back into Eq.~(\ref{bmXiEqn}) and rearranging gives 
\begin{equation*}
\xi_i^{(1)}(t, \bm{\phi}) = \frac{1}{2\left(1 + \frac{\phi_i}{2\alpha}(1 - e^{-\alpha t})\right)^2} 
		\sum_{j = 1}^d Q_{ij} \phi_j \int_0^t \frac{\left(1 + \frac{\phi_i}{2\alpha} (e^{-\alpha\tau} - e^{-\alpha t})\right)^2}
												{1 + \frac{\phi_j}{2\alpha} (e^{-\alpha\tau} - e^{-\alpha t})} d\tau.  
\end{equation*}

The integral is tractable: 
\begin{eqnarray*}
 \int_0^t \frac{\left(1 + \frac{\phi_i}{2\alpha} (e^{-\alpha\tau} - e^{-\alpha t})\right)^2}
												{1 + \frac{\phi_j}{2\alpha} (e^{-\alpha\tau} - e^{-\alpha t})} d\tau 
		& = & \frac{\phi_i^2}{2\alpha \phi_j} \int_0^t \frac{(A + e^{-\alpha \tau})^2}{B + e^{-\alpha \tau}} d\tau \nonumber \\
		& = & \frac{\phi_i^2}{2\alpha^2 \phi_j} \int_{e^{-\alpha t}}^1 \frac{(A + v)^2}{(B + v)v} dv, 
\end{eqnarray*}
where
\[
A = \tfrac{2\alpha}{\phi_i} - e^{-\alpha t}, \qquad B = \tfrac{2\alpha}{\phi_j} - e^{-\alpha t}.  
\]
The indefinite integral is, 
\[
\int \frac{(A + v)^2}{(B + v)v} dv = \frac{1}{B}\{A^2\log v - (A - B)^2 \log(B + v) + Bv \},
\]
and hence the definite integral is 
\[
\frac{\phi_i^2}{2\alpha^2 \phi_j} \left(\frac{1}{B}\left\{ \alpha A^2 t + (A - B)^2 \log\frac{B + e^{-\alpha t}}{B + 1} \right\} + 1 - e^{-\alpha t}\right).  
\]
Reinserting this above and noting that $\sum_{j = 1}^d Q_{ij} = 0$, 
\begin{eqnarray*} 
\xi_i^{(1)}(t, \bm{\phi}) & = &\frac{1}{2\left(1 + \frac{\phi_i}{2\alpha}(1 - e^{-\alpha t})\right)^2} \sum_{j = 1}^d Q_{ij} \times \nonumber \\
	& & \qquad \frac{\phi_i^2}{2\alpha^2} \frac{1}{\frac{2\alpha}{\phi_j} - e^{-\alpha t}} \left\{ \left(\frac{2\alpha}{\phi_i} - e^{-\alpha t} \right)^2 \alpha t 
				+ \left(  \frac{2\alpha}{\phi_i} - \frac{2\alpha}{\phi_j} \right)^2  \log \frac{\frac{2\alpha}{\phi_j}}{\frac{2\alpha}{\phi_j} + 1 - e^{-\alpha t}} \right\}
										\nonumber \\
	& = & \frac{\left(\frac{\phi_i}{2\alpha}\right)^2}{\left(1 + \frac{\phi_i}{\mu(t)}\right)^2} \sum_{j = 1}^d Q_{ij} \frac{1}{\frac{2\alpha}{\phi_j} - e^{-\alpha t}} \times \nonumber \\ 
	& & \qquad\qquad\qquad \left\{ \left(\frac{2\alpha}{\phi_i} - e^{-\alpha t} \right)^2 \alpha t - \left(  \frac{2\alpha}{\phi_i} 
							- \frac{2\alpha}{\phi_j} \right)^2 \log \left(1 + \frac{\phi_j}{\mu(t)}\right) \right\}.  
\end{eqnarray*}

The required LT is then, from Eqs.~(\ref{LTofDensity}) and (\ref{xi0Of_tphi}) 
\begin{eqnarray*}
\zeta(t, \bm{\phi}) & = & \exp\left(-\mathbf{x}(0)\cdot (\bm{\xi}^{(0)}(t, \bm{\phi}) + \theta \bm{\xi}^{(1)}(t, \bm{\phi}) + o(\theta))\right) \nonumber\\
	&  = & \exp\left\{ -\sum_{j = 1}^d x_j(0) \left[\frac{\phi_j}{1 + \phi_j/\mu(t)} + \theta \xi_j^{(1)}(t, \bm{\phi})\right]\right\} \nonumber \\
	& = & e^{-\sum_{j = 1}^d x_j(0) \frac{\phi_j}{1 + \phi_j/\mu(t)}} e^{-\theta \sum_{j = 1}^d x_j(0) \xi_j^{(1)}(t, \bm{\phi})} \nonumber\\
	& = & \zeta^{(0)}(t, \bm{\phi}) \left(1 - \theta \sum_{j = 1}^d x_j(0) \xi_j^{(1)}(t, \bm{\phi}) + \mathcal{O}(\theta) \right),
\end{eqnarray*}
agreeing with Eq.~(\ref{zetaOfTPhi}), where $\xi_i^{(1)}(t, \bm{\phi})$ is given in Eq.(\ref{xi1Of_tphi}).
\qed

%
%

\section{Proof of Theorem~\ref{thm:InverseZeta}: Inversion of the LT $\zeta^{\rm 1-ancestor}(t, \bm{\phi})$}
\label{appendix:invZeta}

This appendix is devoted to calculating the inverse LT of Eq.~(\ref{zeta1Ancestor}).  Set 
\begin{equation}	\label{saDef}
s_i = \frac{\phi_i}{\mu(t)} = \frac{\phi_i}{2\alpha}(1 - e^{-\alpha t}), \quad a = e^{\alpha t} - 1 = 2\alpha\beta(t), 
\end{equation}
and rewrite Eq.~(\ref{zeta1Ancestor}) as 
\begin{equation}	\label{zetaRewritten}
\zeta^{\rm 1-ancestor}(t, \bm{\phi}) = \tilde{F}(\mathbf{s}) = \frac{1}{1 + s_i} + \theta \tilde{F}^{(1)}(\mathbf{s}),
\end{equation}
where, from Eq.~(\ref{xi1Of_tphi}), 
\begin{eqnarray}	\label{fOfSDef}
\tilde{F}^{(1)}(\mathbf{s}) & := & -\frac{\xi_i^{(1)}(t, \mu(t)\mathbf{s})}{\mu(t)} \nonumber \\
	& = & \frac{-a}{2\alpha (1 + s_i)^2} \sum_{j = 1}^d Q_{ij} \frac{s_j}{a - s_j} 
			\left\{ \alpha t \frac{(a - s_i)^2}{a^2} - \frac{(s_j - s_i)^2}{s_j^2} \log\left(1 + s_j \right) \right\}.  
\end{eqnarray}
The inverse LT of $\tilde{F}(\mathbf{s}) = \int_{(\mathbb{R}_{\ge 0})^d} F(\mathbf{u}) e^{-\mathbf{s}\cdot\mathbf{u}} d^d\mathbf{u}$ is 
\begin{equation}	\label{fOfUdef}
F(\mathbf{u}) = e^{-u_i} \prod_{l \ne i} \delta(u_l) + \theta F^{(1)}(\mathbf{u}) 
\end{equation}
where $F^{(1)}(\mathbf{u})$ is the inverse LT of $\tilde{F}^{(1)}(\mathbf{s})$.  From Eq.~(\ref{zetaRewritten}), the inverse LT of 
$\zeta^{\rm 1-ancestor}(t, \bm{\phi})$ is then 
\begin{eqnarray}
f^{\rm 1-ancestor}(t, \mathbf{z}) & = & \mu(t)^d F(\mu(t) \mathbf{z}) \label{f1aInAppendix} \\
	& = & \mu(t)e^{-z_i \mu(t)} \prod_{l \ne i}\delta(z_l) + \theta \mu(t)^d F^{(1)}(\mu(t) \mathbf{z}). \nonumber
\end{eqnarray}

To determine $F^{(1)}(\mathbf{u})$, split the sum in Eq.~(\ref{fOfSDef}) into parts as follows, 
\begin{equation}	\label{tildeFParts}
\tilde{F}^{(1)}(\mathbf{s}) = Q_{ii} \tilde{T}_1(s_i) + \sum_{j \ne i} Q_{ij} \{\tilde{T}_2(s_i, s_j) + \tilde{T}_3(s_i, s_j)\},  
\end{equation}
where 
\[
\tilde{T}_1(s_i) = -\frac{t}{2a} \frac{s_i(a - s_i)}{(1 + s_i)^2}, 
\] 
\[
\tilde{T}_2(s_i, s_j) = -\frac{t}{2a} \frac{(a - s_i)^2}{(1 + s_i)^2} \frac{s_j}{(a - s_j)}, 
\]
and 
\begin{eqnarray*}
\tilde{T}_3(s_i, s_j) & = & \frac{a}{2\alpha(1 + s_i)^2} \frac{1}{a - s_j} \frac{(s_j - s_i)^2}{s_j} \log\left(1 + s_j \right) \nonumber \\
	& = &- \frac{a}{2\alpha(1 + s_i)^2} \left\{ 1 + \left(a - 2s_i + \frac{s_i^2}{a}\right) \frac{1}{s_j -a} - \frac{s_i^2}{a} \frac{1}{s_j} \right\} \log\left(1 + s_j \right). 
\end{eqnarray*}
The inverse LT of $\tilde{F}^{(1)}(\mathbf{s})$ takes the form 
\begin{equation}	\label{inverseF1}
F^{(1)}(\mathbf{u}) = Q_{ii} T_1(u_i) \prod_{l \ne i}\delta(u_l) + \sum_{j \ne i} Q_{ij} \{T_2(u_i, u_j) + T_3(u_i, u_j)\} \prod_{l \ne i,j}\delta(u_l).  
\end{equation}

\begin{table}[t]
\caption{LTs used to invert Eq.~(\ref{tildeFParts}).  $\delta(u)$ is the Dirac delta function, $\gamma_{\rm EM}$ is the Euler-Mascheroni constant, 
$E_1(z) = \int_1^\infty e^{-z\tau} d\tau/\tau$ is the exponential integral and $\lambda > 0$ is any positive real number.  }
\begin{center}
\begin{tabular}{l l}
$F(u)$                                                						& $\tilde{F}(s) = \int_0^\infty F(u) e^{-su} du$			\\
\hline
$\left[a + 2 - (1 + a)u\right]e^{-u} - \delta(u)$ 				& $s(a - s)/(1 + s)^2$   							\\
$(1 + a)\left[(1 + a) u - 2\right]e^{-u} + \delta(u)$				& $(a - s)^2/(1 + s)^2	$							\\
$-ae^{au} - \delta(u)$								& $s/(a - s)$									\\
$\left[(\lambda\theta)^{-1} - \gamma_{\rm EM} \right]\delta(u) 
	- u^{\lambda\theta -1} e^{-u}$						& $\log(1 + s) + \mathcal{O}(\theta)$ as $\theta \to 0$ 	\\
$\left[E_1((1 + a)u) + \log(1 + a)\right]e^{au}$				& $(s - a)^{-1} \log(1 + s)$							\\
$E_1(u)$											& $s^{-1} \log(1 + s)$								\\
$ue^{-u}$											& $1/(1 + s)^2$									\\
$(1 - u)e^{-u}$										& $s/(1 + s)^2$									\\
$(u - 2)e^{-u} + \delta(u)$								& $s^2/(1 + s)^2$								\\
\hline
\end{tabular}
\end{center}
\label{LaplaceTransforms}
\end{table}%

Table~\ref{LaplaceTransforms} lists the LTs used to invert each term.  The inverses of $\tilde{T}_1$ and $\tilde{T}_2$ are straightforward, giving 
\begin{equation}	\label{T1_of_u}
T_1(u_i) = \frac{t}{2a} \left\{\left[(1 + a)u_i - a - 2\right]e^{-u_i} + \delta(u_i)\right\},
\end{equation}
and
\begin{equation}	\label{T2_of_u}
T_2(u_i, u_j) = \frac{t}{2a} \left\{(1 + a)\left[(1 + a) u_i - 2\right]e^{-u_i} + \delta(u_i)\right\}\left\{ae^{au_j} + \delta(u_j)\right\}.
\end{equation}
Calculating $T_3$ requires the inverse LT of $(1 + s_i)^{-2}\log(1 + s_j)$ for $i \ne j$.  However, the inverse LT of $\log(1 + s)$ does not exist {\em per se}.  
In Table~\ref{LaplaceTransforms} we have regularised the inverse LT by exploiting the fact that $\zeta(t, \bm{\phi})$ has only been 
calculated to first order in the overall mutation rate $\theta$, and by making the observation that, to leading order in $\theta$, known solutions to the 
forward Kolmogorov equation for diffusions with mutations  typically behave as $z^{\lambda\theta - 1}$ as the distance from the boundary $z \to 0$, 
for some $\lambda > 0$.  The 2-dimensional inverse LT of the first part of $\tilde{T}_3$ is, to order $\theta$, 
\begin{eqnarray}	\label{LT_T3Part1}
\lefteqn{\text{LT}^{-1}\left( \frac{1}{(1 + s_i)^2} \log\left(1 + s_j \right) \right)}  \nonumber \\ 
	& = & \left[\left(\frac{1}{\lambda_{ij}(u_i) \theta} - \gamma_{\rm EM} \right)\delta(u_j) 
									- u_j^{\lambda_{ij}(u_i)\theta -1} e^{-u_j}\right] u_ie^{-u_i},  
\end{eqnarray}
where the $\lambda_{ij}(u_i) > 0$ for $j \ne i$ are as-yet undetermined finite, analytic functions on $\mathbb{R}_{>0}$.  
A proof of Eq.~(\ref{LT_T3Part1}) follows.  
\begin{proof}
Writing LT$_i$ and LT$_j$ for LTs with respect to $u_i$ and $u_j$ respectively and using LTs from Table~\ref{LaplaceTransforms}, 
\begin{eqnarray*}
\lefteqn{\text{LT}(\text{rhs of Eq.~(\ref{LT_T3Part1})})} \nonumber \\ 
	& = & \text{LT}_i\left(u_ie^{-u_i}\,  \text{LT}_j \left(\left(\frac{1}{\lambda_{ij}(u_i) \theta} - \gamma_{\rm EM} \right)\delta(u_j) 
									- u_j^{\lambda_{ij}(u_i)\theta -1} e^{-u_j}\right) \right) \nonumber \\
	& = & \text{LT}_i\left(u_ie^{-u_i} \left( \log\left(1 + s_j \right) + \mathcal{O}(\theta)\right)\right) \nonumber \\
	& = & \text{LT}_i\left(u_ie^{-u_i}\right) \log\left(1 + s_j \right) + \mathcal{O}(\theta) \nonumber \\
	& = & \frac{1}{(1 + s_i)^2} \log\left(1 + s_j\right) + \mathcal{O}(\theta), \qquad \text{as }\theta \to 0.  
\end{eqnarray*}
\end{proof}

Inverting the remaining parts of $\tilde{T}_3$ is a straightforward application of LTs from Table~\ref{LaplaceTransforms}, yielding   
\begin{eqnarray}	\label{T3_of_u}
T_3(u_i, u_j) & = & - \frac{a}{2\alpha} \left\{\left[\left(\frac{1}{\lambda_{ij}(u_i) \theta} - \gamma_{\rm EM} \right)\delta(u_j) 
									- u_j^{\lambda_{ij}(u_i)\theta -1} e^{-u_j}\right] u_ie^{-u_i} \right. \nonumber \\
	& & \qquad  + \left[au_ie^{-u_i} - 2(1 - u_i)e^{-u_i} + \frac{1}{a}(u_i - 2)e^{-u_i} + \frac{1}{a}\delta(u_i) \right]  \nonumber \\
	& & \qquad\qquad \times \left[E_1((1 + a)u_j) + \log(1 + a)\right]e^{au_j}	\nonumber \\ 
	& & \qquad - \left. \left[\frac{1}{a}(u_i - 2)e^{-u_i} + \frac{1}{a}\delta(u_i)\right] E_1(u_j) \right\} \nonumber \\
	& = & - \frac{1}{2\alpha} \left\{a\left[\left(\frac{1}{\lambda_{ij}(u_i) \theta} - \gamma_{\rm EM} \right)\delta(u_j) 
									- u_j^{\lambda_{ij}(u_i)\theta -1} e^{-u_j}\right] u_i\right. \nonumber \\
	& & \qquad  +\, (1 + a)\{(1 + a)u_i - 2\}  \left[E_1((1 + a)u_j) + \log(1 + a) \right]e^{au_j} \nonumber \\ 
	& & \qquad   \left.\vphantom{\frac{X}{Y}} - \, (u_i - 2)  E_1(u_j) \right\} e^{-u_i}  \nonumber \\
	& & \qquad   -\, \frac{1}{2\alpha} \left\{[E_1((1 + a)u_j) + \log(1 + a)]e^{au_j} - E_1(u_j) \right\} \delta(u_i) .
\end{eqnarray}

Substituting Eqs.~(\ref{T1_of_u}) and (\ref{T2_of_u}) into Eq.~(\ref{inverseF1}) and making use of the property $\sum_{j = 1}^d Q_{ij} = 0$ to cancel 
terms containing $\prod_{l = 1}^d \delta(u_l)$, the first two terms of $F^{(1)}(\mathbf{u})$ are 
\begin{eqnarray*}
\lefteqn{Q_{ii} T_1(u_i) \prod_{l \ne i}\delta(u_l) + \sum_{j \ne i} Q_{ij} T_2(u_i, u_j) \prod_{l \ne i,j}\delta(u_l)} \nonumber \\
	& = & \frac{t}{2a} \left\{ Q_{ii} [(1 + a)u_i - a - 2]e^{-u_i} \prod_{l \ne i} \delta(u_l) \right. \nonumber \\
	& & \qquad\quad + \sum_{j \ne i} Q_{ij} (1 + a)[(1 + a)u_i - 2] e^{-u_i} ae^{au_j} \prod_{l \ne i, j} \delta(u_l) \nonumber \\
	& & \qquad\quad + \sum_{j \ne i} Q_{ij} (1 + a)[(1 + a)u_i - 2] e^{-u_i} \prod_{l \ne i} \delta(u_l) \nonumber \\
	& & \qquad\quad \left. + \sum_{j \ne i} Q_{ij} ae^{au_j} \prod_{l \ne j} \delta(u_l) \right\} 
\end{eqnarray*}
\begin{eqnarray}	\label{QT1and2}
	& = & \frac{t}{2} \left\{ Q_{ii} [1 - (1 + a)u_i] e^{-u_i} \prod_{l \ne i} \delta(u_l) \right. \nonumber \\
	& & \qquad\quad  + \sum_{j \ne i} Q_{ij} e^{au_j} \prod_{l \ne j} \delta(u_l) \nonumber \\
	& & \qquad\quad \left. + \sum_{j \ne i} Q_{ij} (1 + a)[(1 + a)u_i - 2] e^{-u_i + au_j} \prod_{l \ne i, j} \delta(u_l) \right\}	 \nonumber \\
	& = & \frac{t}{2} \left\{ Q_{ii} (1 - u_i e^{\alpha t}) e^{-u_i} \prod_{l \ne i} \delta(u_l) \right. \nonumber \\
	& & \qquad\quad  + \sum_{j \ne i} Q_{ij} e^{au_j} \prod_{l \ne j} \delta(u_l) \nonumber \\
	& & \qquad\quad \left. + \sum_{j \ne i} Q_{ij} e^{\alpha t}(u_i e^{\alpha t} - 2) e^{-u_i + au_j} \prod_{l \ne i, j} \delta(u_l) \right\}.	
\end{eqnarray}

The remaining, $T_3$-dependent, term of $F^{(1)}(\mathbf{u})$ contains a piece proportional to $\theta^{-1}$ which contributes to the zeroth order 
in $\theta$ part of $F(\mathbf{u})$ as follows (see Eqs.~(\ref{fOfUdef}), (\ref{inverseF1}) and (\ref{T3_of_u}))
\begin{equation}	\label{OrderZeroPartF}
F(\mathbf{u}) = \left[ e^{-u_i}  - \frac{a}{2\alpha} u_i e^{-u_i} \sum_{j \ne i} \frac{Q_{ij}}{\lambda_{ij}(u_i)} \right] \prod_{l \ne i}\delta(u_l) + \mathcal{O}(\theta).
\end{equation}
The contribution to the coefficient of $\theta$ in $F(\mathbf{u})$ from $T_3(u_i, u_j)$ is then 
\begin{eqnarray}	\label{QT3}
\lefteqn{\sum_{j \ne i} Q_{ij} \{T_3(u_i, u_j) - \text{the term containing }\theta^{-1}\} \prod_{l \ne i,j}\delta(u_l)} \nonumber \\
	& = & -\frac{a}{2\alpha} Q_{ii} \gamma_{\rm EM} u_ie^{-u_i}  \prod_{l \ne i} \delta(u_l)  \nonumber \\ 
	& & -\, \frac{1}{2\alpha} \sum_{j \ne i} Q_{ij} \left\{[E_1((1 + a)u_j) + \log(1 + a)]e^{au_j} - E_1(u_j) \right\} \prod_{l \ne j} \delta(u_l)  \nonumber \\ 
	& & +\, \frac{1}{2\alpha} \sum_{j \ne i} Q_{ij} \left\{a u_j^{\lambda_{ij}(u_i)\theta -1} e^{-u_j} u_i \right.   \nonumber \\ 
	& & \qquad\qquad -\, (1 + a)\{(1 + a)u_i - 2\}  \left[E_1((1 + a)u_j) + \log(1 + a) \right]e^{au_j} \nonumber \\
	& & \qquad\qquad\qquad\qquad\qquad\qquad  \left.\phantom{u_j^{\lambda_{ij}}}    + \, (u_i - 2)  E_1(u_j) \right\} e^{-u_i} \prod_{l \ne i,j} \delta(u_l) \nonumber \\
	& = & \frac{1 - e^{\alpha t}}{2\alpha} Q_{ii} \gamma_{\rm EM} u_ie^{-u_i}  \prod_{l \ne i} \delta(u_l)  \nonumber \\ 
	& & -\, \frac{1}{2\alpha} \sum_{j \ne i} Q_{ij} \left\{[\alpha t + E_1(u_je^{\alpha t})]e^{au_j} - E_1(u_j) \right\} \prod_{l \ne j} \delta(u_l)  \nonumber \\ 
	& & +\, \frac{1}{2\alpha} \sum_{j \ne i} Q_{ij} \left\{(e^{\alpha t} - 1) u_j^{\lambda_{ij}(u_i)\theta -1} e^{-u_j} u_i \right.   \nonumber \\ 
	& & \qquad\qquad -\, e^{\alpha t}(u_i e^{\alpha t} - 2)  \left[\alpha t + E_1(u_j e^{\alpha t})\right]e^{au_j} \nonumber \\
	& & \qquad\qquad\qquad\qquad\qquad\qquad  \left.\phantom{u_j^{\lambda_{ij}}}    + \, (u_i - 2)  E_1(u_j) \right\} e^{-u_i} \prod_{l \ne i,j} \delta(u_l) \nonumber \\
\end{eqnarray}

Finally we insert Eqs~(\ref{QT1and2}), (\ref{OrderZeroPartF}) and (\ref{QT3}) into Eq.~(\ref{fOfUdef}) and the inverse LT of Eq.~(\ref{tildeFParts}) to obtain the 
inverse LT to order $\theta$ in the form of a sum over line and surface densities: 
\begin{eqnarray*}
\lefteqn{F(\mathbf{u}) = F_0(u_i) \prod_{l \ne i} \delta(u_l) + \theta\left(F^{\text{line-par}}(u_i) \prod_{l \ne i} \delta(u_l) 
							+ \sum_{j \ne i} F_j^\text{line-perp}(u_j) \prod_{l \ne j} \delta(u_l) \right.}  \nonumber \\
	& & \qquad\qquad\qquad\qquad\qquad\qquad \left. + \sum_{j \ne i} F_j^\text{surface}(u_i, u_j) \prod_{l \ne i,j} \delta(u_l)\right) + \mathcal{O}(\theta^2),
\end{eqnarray*}
where
\begin{align*} 
F_0(u_i) &= e^{-u_i}  - \frac{a}{2\alpha} u_i e^{-u_i} \sum_{j \ne i} \frac{Q_{ij}}{\lambda_{ij}(u_i)}, \nonumber \\ 
F^{\text{line-par}}(u_i) & = \frac{1}{2\alpha} Q_{ii} \left\{ \alpha t (1 - u_i e^{\alpha t}) + (1 - e^{\alpha t})\gamma_{\rm EM} u_i \right\} e^{-u_i}, \nonumber \\
F_j^{\text{line-perp}}(u_j) &= \frac{1}{2\alpha} Q_{ij} \left\{E_1(u_j) - E_1(u_j e^{\alpha t}) e^{ au_j}\right\}, \nonumber \\ 
F_j^\text{surface}(u_i, u_j) &= \frac{1}{2\alpha} Q_{ij} \left\{(e^{\alpha t} - 1) u_j^{\lambda_{ij}(u_i)\theta -1} e^{-u_j} u_i \phantom{u_j^{\lambda_{ij}}} \right. \nonumber \\
						& \qquad\quad \left.\phantom{u_j^{\lambda_{ij}}} -e^{\alpha t}(u_i e^{\alpha t} - 2)E_1(u_j e^{\alpha t}) e^{ au_j} + (u_i - 2)E_1(u_j) \right\}  e^{-u_i}, \nonumber \\
\end{align*}
$a$ is defined by Eq.~(\ref{saDef}) and the $\lambda_{ij}(u_i) > 0$ for $j \ne i$ are as-yet undetermined functions on $\mathbb{R}_{>0}$.  
Finally insert $F(\mathbf{u})$ into Eq.~(\ref{f1aInAppendix}) to obtain the density $f^{\text{1-ancestor}}(t, \mathbf{z})$ to $\mathcal{O}(\theta)$.  
\qed

%
%

\section{Integrals required for Section~\ref{sec:XMoments}}
\label{appendix:MomentIntegrals}

The following integrals are required for confirming the normalisation of $f^{\text{1-ancestor}}(t, \mathbf{z})$ and for calculating moments of $\mathbf{X}(t)$.  
\begin{align}
\int_0^\infty u^n E_1(u) du 							&= \frac{n!}{n + 1}, \quad n \ge 0,  	\label{intglunE1} \\
\int_0^\infty u^n e^{(e^{\alpha t} - 1)u} E_1(ue^{\alpha t}) du 	&= \begin{cases}
					\frac{-\alpha t}{(1 - e^{\alpha t})},							& n = 0; \\
					\frac{-n!}{(1 - e^{\alpha t})^{n + 1}} \left(\alpha t + \sum_{l = 1}^n \frac{(1 - e^{\alpha t})^l}{l} \right), & n \ge 1 \\
					\end{cases}					\label{intglAsSum}\\
												& = \frac{n!}{n + 1}\,_2F_1(1, n + 1; n + 2; 1 - e^{\alpha t}), \quad n \ge 0, \label{intgl2F1} \\
\int_0^\infty u^{\lambda\theta + n - 1} e^{-u} du 				&= \begin{cases}
					\frac{1}{\lambda\theta} - \gamma_{\rm EM} + \mathcal{O}(\theta), 	&n = 0; \\
					(n - 1)!  + \mathcal{O}(\theta), 								& n \ge 1.
					\end{cases} 					\label{uLambdaTheta}
\end{align}

Eq.~(\ref{intglunE1}) follows using elementary methods from  the definition $E_1(z) = \int_1^\infty e^{-z\tau} d\tau/\tau$ of the exponential integral. 
To derive Eq.~(\ref{intglAsSum}), begin with 
\begin{align*}
\int_0^\infty u^n e^{(e^{\alpha t} - 1)u} E_1(ue^{\alpha t}) du 	&= \int_0^\infty u^n e^{(C - 1)u} \int_1^\infty \frac{e^{-Cu\tau}}{\tau} d\tau\, du \\
	&= \int_1^\infty \frac{1}{\tau} \int_0^\infty u^n c^{-(C\tau + 1 - C)u} du\, d\tau \\
	&= \frac{n!}{C^{n + 1}} \int_1^\infty \frac{d\tau}{\tau(\tau + A)^{n + 1}},  
\end{align*}
where $C = e^{\alpha t}$ and $A = (1 - C)/C = e^{-\alpha t} - 1$.  The partial fraction decomposition is 
\[
\frac{1}{\tau(\tau + A)^{n + 1}} = \frac{1}{A}\left(\frac{1}{A^n\tau} - \sum_{k = 0}^n \frac{1}{A^k (\tau + A)^{n + 1 - k}} \right), \qquad n \ge 0,
\]
as can be checked by induction.  Hence the integral is 
\begin{align*}
\int_1^\infty \frac{d\tau}{\tau(\tau + A)^{n + 1}} &= \frac{1}{A^{n + 1}} \int_1^\infty \left(\frac{1}{\tau} - \frac{1}{\tau + A} \right) d\tau 
					- \frac{1}{A} \sum_{k = 0}^{n - 1} \frac{1}{A^k} \int_1^\infty \frac{d\tau}{(\tau + A)^{n + 1 - k}} \\ 
		&= \frac{1}{A} \left\{ \frac{1}{A^n} \log(1 + A) - \sum_{l = 1}^{n} \frac{1}{l} \frac{1}{(1 + A)^l A^{n - l}} \right\},  
\end{align*}
with the sum over $l$ absent when $n = 0$.  
Alternatively, begin with \cite[Eq.~2.117.3]{Gradshteyn65} and use induction.  Resubstituting for $C$ and $A$ then gives Eq.~(\ref{intglAsSum}).  

For moderately large $n$, numerical evaluation of Eq.~(\ref{intglAsSum}) is subject to roundoff errors because it is a small difference of two numbers.  
A numerically stable form is Eq.~(\ref{intgl2F1}).  Writing $z = 1 - e^{\alpha t}$ and using the identity $\log(1 - z) = -\sum_{l = 1}^\infty z^l/l$, 
Eq.~(\ref{intglAsSum}) is 
\begin{align*}
\frac{-n!}{z^{n + 1}}\left( \log(1 - z) + \sum_{l = 1}^n \frac{z^l}{l}\right) &= \frac{n!}{z^{n + 1}}\sum_{l = n + 1}^\infty \frac{z^l}{l} \\
	&= n!\sum_{k = 0}^\infty \frac{z^k}{k + 1 + n} \\
	&= \frac{n!}{n + 1} \frac{\Gamma(n + 2)}{\Gamma(1)\Gamma(n + 1)} 
					\sum_{k = 0}^\infty \frac{\Gamma(1 + k)\Gamma(n + 1 + k)}{\Gamma(n + 2 + k)} \frac{z^k}{k!} \\
	&= \frac{n!}{n + 1} \,_2F_1(1, n + 1; n + 2; z), 
\end{align*}
using the definition \citet[Eq.~(15.1.1)]{Abramowitz:1965sf} of the ordinary hypergeometric function in the last line. This agrees with Eq.~(\ref{intgl2F1}).  
Eq.~(\ref{uLambdaTheta}) follows by expanding the Gamma function $\Gamma(z) = \int_0^\infty u^{z - 1} e^{-u} du$ about $z = n$.  

%
%

\section{Consistency with sub-critical moments to first order in $\theta$}
\label{appendix:QSmoments}

The non-zero quasi-stationary moments calculated calculated to first order in $\theta$ from Eq.~(\ref{EofXforQS}) are, 
starting from Eqs.~(\ref{momentsUtoX}), (\ref{EXim}) and (\ref{EXjn}), 
\begin{align*}
\mathbb{E}_{\rm QS}[X_i^m] &= \lim_{t \to \infty} \beta(t; -|\alpha|) \left( \pi_i \mathbb{E}[U_i(t)^m \mid \mathcal{E}_i] 
										+ \sum_{j \ne i} \pi_j \mathbb{E}[U_i(t)^n \mid \mathcal{E}_j] \right) \nonumber \\
		&= \frac{1}{2|\alpha|} \lim_{t \to \infty} \left( \pi_i m!\left[1 - \frac{\theta}{2|\alpha|}Q_{ii} 
				\left\{ - \frac{1 - me^{-|\alpha| t}}{1 - e^{-|\alpha| t}}  |\alpha| t - m + 1 \right\} \right] \right. \nonumber \\ 
		&  \qquad\qquad - \sum_{j \ne i} \pi_j \frac{\theta}{2|\alpha|} Q_{ij} (e^{-|\alpha| t} - 1)(m - 1)! \times  \\
		& \qquad\qquad\qquad \left. \left\{1 - \frac{m}{(1 - e^{-|\alpha| t})^m}
										 \left[ -|\alpha| t + \sum_{l = 1}^m \frac{(1 - e^{-|\alpha| t})^l}{l} \right] \right\} \right) \nonumber \\
		&= \frac{1}{2|\alpha|} \lim_{t \to \infty} \left( \pi_i m!\left[1 - \frac{\theta}{2|\alpha|}Q_{ii} \left\{ - |\alpha| t - m + 1 \right\} \right] \right. \nonumber \\ 
		&  \qquad\qquad + \left. \pi_i \frac{\theta}{2|\alpha|} Q_{ii} (- 1)(m - 1)!  \left\{1 - m \left[ -|\alpha| t + H_m \right] \right\} \right) \nonumber \\
		&= \frac{1}{2|\alpha|} \pi_i m! + \frac{\theta}{4|\alpha|^2} \pi_i Q_{ii} (m - 1)! (m^2- m - 1 + mH_m), \quad m \ge 1; i = 1, \ldots d, 
\end{align*} 
where $H_m = \sum_{l = 1}^m l^{-m}$ is the $m$th harmonic number for $m \ge 1$, and from Eqs.~(\ref{momentsUtoX}) and (\ref{EXimXjn}), 
\begin{eqnarray*}
\lefteqn{\mathbb{E}_{\rm QS}[X_i^{n_i} X_j^{n_j}] } \nonumber \\
		&=& \lim_{t \to \infty} \beta(t; -|\alpha|)^2 \left( \pi_i \mathbb{E}[U_i(t)^{n_i} U_j(t)^{n_j} \mid \mathcal{E}_i] + (i \leftrightarrow j)\right) \nonumber \\
		&=& \frac{1}{4|\alpha|^2} \lim_{t \to \infty} \left( -\pi_i \frac{\theta}{2|\alpha|} Q_{ij} (n_j - 1)!n_i! 
							\left\{(n_i + 1)e^{-|\alpha| t} - \frac{n_i + 2n_j + 1}{n_j + 1} \right.\right.   \nonumber \\
		&& 	\qquad \left.\left. - \, (2 - (n_i + 1)e^{-|\alpha| t}) \frac{n_je^{-|\alpha| t}}{(1 - e^{-|\alpha| t})^{n_j + 1}} 
							\left( -|\alpha| t + \mathcal{O}(1) \right) \right\} + (i \leftrightarrow j) \right)     \nonumber \\
		& = & \frac{1}{8|\alpha|^3} \pi_i Q_{ij} \frac{(n_j - 1)!n_i!}{n_j + 1}(n_i - 2n_j - 1) + (i \leftrightarrow j), \quad  n_i \ne n_j; n_i, n_j \ge 1; i \ne j,
\end{eqnarray*} 
where $(i \leftrightarrow j)$ means the previous term with the indices $i$ and $j$ interchanged.  These moments agree with 
\citet[Theorem~3 and Eq.~(29)]{BurdenGriffiths23}, whose results are primarily stated for the case $\alpha = \tfrac{1}{2}$.  

%
%

\bibliographystyle{elsarticle-harv}\biboptions{authoryear}

\end{document}